\documentclass[11pt,letterpaper, reqno]{amsart}

\usepackage{enumerate}
\usepackage{amssymb,amsmath,amsthm,mathrsfs,enumerate,graphicx, color}
\usepackage[
colorlinks=true,
linkcolor=black,
citecolor=black,
urlcolor=black,
]{hyperref}
\usepackage{esint}

\usepackage{amssymb}
\usepackage{amsthm}
\usepackage{amsxtra}
\usepackage{graphicx}
\usepackage{mathabx}

\def\R{{\mathbb R}}
\def\Z{{\mathbb Z}}

\newcommand{\la}{\langle}
\newcommand{\ra}{\rangle}

\newtheorem{theorem}{Theorem}

\newtheorem{proposition}[theorem]{Proposition}
\newtheorem{lemma}[theorem]{Lemma}

\theoremstyle{remark}
\newtheorem{remark}[theorem]{Remark}

\numberwithin{equation}{section}

\numberwithin{theorem}{section}

\numberwithin{table}{section}

\numberwithin{figure}{section}

\def\eqnn{\begin{eqnarray*}}
\def\eeqnn{\end{eqnarray*}}
\def\eqn{\begin{eqnarray}}
\def\eeqn{\end{eqnarray}}

\ifx\pdfoutput\undefined
  \DeclareGraphicsExtensions{.pstex, .eps}
\else
  \ifx\pdfoutput\relax
    \DeclareGraphicsExtensions{.pstex, .eps}
  \else
    \ifnum\pdfoutput>0
      \DeclareGraphicsExtensions{.pdf}
    \else
      \DeclareGraphicsExtensions{.pstex, .eps}
    \fi
  \fi
\fi

\title{Global Well-posedness of the NLS System for infinitely many fermions}
\date{\today}
\author{Thomas Chen}
\address{T. Chen,  
Department of Mathematics, University of Texas at Austin.}
\email{tc@math.utexas.edu}
\author{Younghun Hong}
\address{Y. Hong,  
Department of Mathematics, University of Texas at Austin.}
\email{yhong@math.utexas.edu}
\author{Nata\v{s}a Pavlovi\'c}
\address{N. Pavlovi\'c,  
Department of Mathematics, University of Texas at Austin.}
\email{natasa@math.utexas.edu}
\begin{document}

\maketitle

\begin{abstract}
In this paper, we study the mean field quantum fluctuation dynamics for  a system of
infinitely many fermions with delta pair interactions 
in the vicinity of an equilibrium solution (the Fermi sea) at zero temperature, in dimensions $d=2,3$,
and prove global well-posedness of the corresponding Cauchy problem.
Our work extends some of the recent important results obtained by M. Lewin and J. Sabin in
\cite{LS1,LS2}, who addressed this problem for more regular pair
interactions. 
\end{abstract}

\tableofcontents

\section{Introduction}

There are two fundamental types of elementary particles in nature, bosons and fermions.
The dynamics of a system of $N$
bosons in ${\R}^d$ interacting pairwise is described by solutions to the
Schr\"odinger equation on $L^2(\R^{Nd})$
\begin{equation}\label{eq-Schrod-1}
    i\partial_t\psi_N = H_N \psi_N
    \;\;,\;\;\psi_N(0)=\psi_{N,0}
\end{equation}
where  
$$
	H_N=\sum_{j=1}^N(-\Delta_{x_j})+\sum_{1\leq j<k\leq N}U(x_j-x_k)
$$ 
is the $N$-particle Hamiltonian; the
precise form of the pair interaction potential $U$ depends on the model at hand. 
Bosons satisfy Bose-Einstein statistics, that is,
the bosonic wave function $\psi_N$ is completely symmetric
under exchange of particle variables $x_j\in\R^d$, i.e.,
$$
    \psi_N(x_1,\dots,x_N)=\psi_N(x_{\sigma(1)},\dots,x_{\sigma(N)}) \,,
$$
for any permutation $\sigma\in S_N$ (the $N$-th symmetric group). 
The simplest example of a completely symmetric state is a factorized state
$$
    \psi(x_1,\dots,x_N) = \prod_{j=1}^N \phi(x_i) \,.
$$
It describes a configuration where all bosons are in the same quantum mechanical state $\phi$.
This situation prominently occurs in Bose-Einstein condensation, a topic in mathematical research 
that has been extremely actively investigated in recent years, \cite{ALSSY,LSY}.
For a bosonic system with mean field scaling $U(x)=\frac1N V(x)$, 
it can be proven rigorously  (for $V$ sufficiently regular) that in the limit $N \rightarrow \infty$,
the mean field dynamics is characterized by a factorized state where $\phi$ satisfies a
nonlinear Hartree equation, 
$$
	i\partial_t\phi=-\Delta\phi +(V*|\phi|^2)\phi.
$$ 
If $V=V_N$  depends on $N$ and
tends, in a suitable manner, to a delta distribution $\delta$ as $N\rightarrow\infty$,
it can be proven that
the mean field dynamics is determined by a 
nonlinear  Schr\"{o}dinger equation, 
$$
i\partial_t\phi=-\Delta\phi +|\phi|^2\phi,
$$
see  \cite{CxH,CP2,CP3,CHPS1,ESY1,ESY2,FGS,He,KSS,LNR,KM2,Sp} 
and the references therein, and \cite{Sc1} for a survey.

In this paper, we will study the mean field dynamics of a system of infinitely 
many fermions, and are particularly interested in the situation where the pair
interaction between particles are described by a delta potential.
Fermions satisfy Fermi-Dirac statistics; while
their dynamics is determined by a Schr\"odinger equation of the same form as \eqref{eq-Schrod-1},
the wave function $\psi_N$ is completely antisymmetric
under exchange of particle variables,
$$
    \psi_N(x_1,\dots,x_N)=(-1)^{sign(\sigma)}
    \psi(x_{\sigma(1)},\dots,x_{\sigma(N)})
$$
for any permutation $\sigma\in S_N$. The simplest example of such a state is
a Slater determinant
$$
    \psi(x_1,\dots,x_N) = {\rm det}\Big[u_i(x_j)\Big]_{i,j=1}^N \,.
$$
As a consequence of the antisymmetry, no two fermions can be in same quantum state $u_j$,
otherwise $\psi_N$ vanishes; in fact,
$\{u_j\}$ can generally be chosen to be an orthonormal family in $L^2(\R^d)$. 
This is the Pauli exclusion principle.

The ground state for a system of $N$ non-interacting fermions on the unit torus ${\mathbb T}^d$ 
with kinetic energy operator
$\sum_j(-\Delta_{x_j})$ has the following simple form.
In frequency space, 
$$
    \Big(\psi_N \;,\; \sum_j(-\Delta_{x_j})\psi_N\Big)  = \sum_{p_i\in\Z^d, i=1,...,N}\Big(
    \sum_{j=1}^N p_j^2 \Big)\;|\widehat{\psi_N}(p_1,\dots,p_N)|^2 \,.
$$
Therefore, it is easy to see that the lowest energy configuration is given by 
$\widehat{\psi_N}=\bigwedge_{i\in F_N}\delta_{i}$ where $F_N\subset\Z^d$ is the subset of the $N$ distinct lattice points  
closest to the origin, and $\delta_i$ is the Kronecker delta at $i\in\Z^d$. Clearly, 
the set $F_N$ is an approximate ball $B_{R_N}(0)\cap{\mathbb Z}^d$ of radius $R_N \sim N^{1/d}$,
and it is referred to as the {\em Fermi sea}. 

In this paper, we study the mean field quantum fluctuation dynamics for  a system of
infinitely many fermions with delta pair interactions 
in the vicinity of the Fermi sea at zero temperature, in dimensions $d=2,3$,
and prove global well-posedness of the Cauchy problem.
Our work extends some recent important results obtained by M. Lewin and J. Sabin,
\cite{LS1,LS2}, who address this problem for more regular pair
interactions. 

To describe the problem in more detail, we first consider a finite system of $N$ fermions with  
mean field interactions involving a pair potential $w$.
This system is described by $N$ coupled Hartree equations 
\begin{equation}\label{NLHsysN}
\left\{\begin{aligned}
i\partial_t u_1&=(-\Delta+w * \rho)u_1 \quad,& \quad\quad  u_1(t=0)&=u_{1,0}\\
               &\cdots &   & \cdots\\
i\partial_t u_N&=(-\Delta+w * \rho)u_N \quad,& \quad\quad  u_N(t=0)&=u_{N,0}
\end{aligned}
\right.
\end{equation}
where $\rho$ is the total density of particles
\begin{equation}\label{NLHsysN-rho}
\rho(t,x)=\sum_{j=1}^N|u_j(t,x)|^2
\end{equation}
and $w$ is the interaction potential.
To account for the Pauli principle, the family $\{u_{j,0}\}_{j=1}^N$ 
is assumed to be orthonormal, and
it can be shown that  $\{u_{j}\}_{j=1}^N$ remains orthonormal
for $t>0$, as long as the Cauchy problem is well-posed in an appropriate space of solutions.

Systems similar to \eqref{NLHsysN}, but with a finite expected particle number
$\int \rho dx$ as $N\rightarrow\infty$,
have been analyzed extensively in the literature, see for instance
\cite{ACV1,BdPF-1, BdPF-2, BM1,Ch-1976, Z-1992}. Those systems describe
a dilute gas with normalized density 
$\rho(t,x)=\frac1N\sum_{j=1}^N|u_j(t,x)|^2$ 
(or $\rho(t,x)=\sum_{j=1}^\infty \lambda_j|u_j(t,x)|^2$ with $\lambda_j>0$ and $\sum\lambda_j =1$) 
such that in the limit as $N\rightarrow\infty$, $\gamma$ is trace class.
They can be derived rigorously from
a combined mean field and semi-classical limit
from a quantum system of interacting fermions; see 
\cite{BGGM-2003, BEGMY-2002, EESY-2004, FK-2011, BNS,NS} 
and the references therein. 
In this limit, 
the exchange term is of a lower order of magnitude in powers of $\frac1N$
than the main interaction term, therefore it 
does not appear in this analysis.

Corresponding to \eqref{NLHsysN}, one can introduce the one particle 
density matrix 
\begin{equation} \label{intro-gamma} 
\gamma_N(t) = \sum_{j=1}^N |u_{j}(t) \rangle\langle u_{j}(t)|,
\end{equation} 
which is the rank-$N$ orthogonal projection onto the 
span of the orthonormal family $\{u_{j}(t)\}_{j=1}^N$.
Then the system \eqref{NLHsysN} can be written in the density matrix form
\begin{align} 
    & i\partial_t\gamma_N = [-\Delta, \gamma_N] + [w * \rho_{\gamma_N}, \gamma_N]  \label{NLHsys1-N} \\
    & \gamma_N(t=0) =  \sum_{j=1}^N |u_{j,0}\rangle\langle u_{j,0}|, \label{NLHsys1-id-N}
\end{align}
%with initial data $\gamma_0$ corresponding to $N$ fermions 
%whose wave functions are given by the orthonormal family  $\{u_{j,0}\}_{j=1}^N$ in $L^2$;
%orthogonality accounts for the fermion statistics. 
with density
\begin{equation} \label{NLHsys1-rho-N}
    \rho_{\gamma_N}(t,x) = \gamma_N(t,x,x) \,.
\end{equation} 
Orthonormality of the family $\{u_{j}\}_{j=1}^N$  implies that $0\leq\gamma\leq1$.

For the system \eqref{NLHsysN} - \eqref{NLHsysN-rho}, respectively \eqref{NLHsys1-N} - \eqref{NLHsys1-rho-N},
the expected particle number $\int\rho_N dx$ diverges in the limit $N\rightarrow\infty$,
hence
the one-particle density matrix $\gamma=\sum_{j=1}^\infty|u_j\rangle\langle u_j|$
fails to be of trace class (but has bounded operator norm). 
The analysis of the dynamics of the fermion gas in this case becomes much more difficult.
Lewin and Sabin, in \cite{LS1,LS2}, were the first authors to address the behavior of 
\eqref{NLHsysN} in this situation. The main problem is to properly give meaning to,
and to understand
solutions to  
the evolution equation 
\begin{align} 
    & i\partial_t\gamma = [-\Delta, \gamma] + [w * \rho_{\gamma}, \gamma]  \label{NLHsys1} \\
    & \gamma(t=0) =  \gamma_0, \label{NLHsys1-id} 
\end{align}
%with initial data $\gamma_0$ corresponding to $N$ fermions 
%whose wave functions are given by the orthonormal family  $\{u_{j,0}\}_{j=1}^N$ in $L^2$;
%orthogonality accounts for the fermion statistics. 
with $\rho_{\gamma}(t,x) = \gamma(t,x,x)$,
where $\gamma_0$ is not of trace class.
To be consistent with
the fact that the system models the behavior of infinitely many fermions, the Pauli principle
is incorporated by requiring that $0 \leq \gamma(0) \leq 1$, 
thus $\gamma$ has a bounded operator norm
(among other properties). We note that the exchange term is in this situation
again of lower order, and is therefore omitted.

In  \cite{LS1,LS2}, Lewin and Sabin study 
the dynamics of trace class perturbations $Q:=\gamma-\gamma_f$ 
around a non-trace class reference state $\gamma_f$.
The latter is chosen to 
corresponds to the Fermi sea of the non-interacting system. 
For inverse temperature $\beta>0$ and chemical 
potential $\mu>0$, $\gamma_f$ is given by the Fermi-Dirac distribution
\eqn 
    \gamma_f(x,y) = \int_{\R^d} \frac{e^{ip(x-y)}}{e^{\beta(p^2-\mu)}+1}  dp = 
    \Big(\frac{1}{e^{\beta(-\Delta-\mu)}+1}\Big)(x,y) \,.
\eeqn
while in the zero temperature limit,
\eqn
    \gamma_f = \Pi_\mu^-=\mathbf{1}_{(-\Delta\leq\mu)} \,.
\eeqn
Lewin and Sabin prove that the Cauchy problem for  
$Q$ is globally well-posed in a suitable subspace
of the space of trace class operators, provided that the pair interaction $w$
is sufficiently regular. 

In the work at hand, we extend the results of \cite{LS1,LS2} to the most singular case 
$w=\delta$, so that the potential term becomes $\delta*\rho=\rho$.
The finite $N$ analogue to \eqref{NLHsysN} is the system of coupled NLS equations,
\begin{equation}\label{NLSsysN}
\left\{\begin{aligned}
i\partial_t u_1&=(-\Delta+\rho)u_1 \quad,& \quad\quad  u_1(t=0)&=u_{1,0}\\
               &\cdots &   & \cdots\\
i\partial_t u_N&=(-\Delta+\rho)u_N \quad,& \quad\quad  u_N(t=0)&=u_{N,0}
\end{aligned}
\right.
\end{equation}
where
\begin{equation}
\rho(t,x)=\sum_{j=1}^N|u_j(t,x)|^2.
\end{equation}
The family  $\{u_{j}\}_{j=1}^N$ remains orthonormal for $t>0$, as long as 
\eqref{NLSsysN} is well-posed.

To study the system in the limit $N\rightarrow\infty$, we employ the density matrix 
formalism as in  \cite{LS1,LS2}, and consider \eqref{NLHsys1} with $w\rightarrow\delta$,
\begin{align}
    & i\partial_t\gamma = [-\Delta,\gamma] + [\rho,\gamma] \label{NLSsys1} \\
    & \gamma(t=0) = \gamma_0 . \label{NLSsys1-id} 
\end{align}
Again, the Pauli principle requires that $0 \leq \gamma_0 \leq 1$.
For simplicity, we consider a reference state $\gamma_f$ that corresponds to the
Fermi sea of the non-interacting system at zero temperature, and chemical potential $\mu>0$, 
\eqn
    \gamma_f = \Pi_\mu^-=\mathbf{1}_{(-\Delta\leq\mu)} \,.
\eeqn
We study perturbations
\eqn 
    Q = \gamma - \Pi_\mu^- \,,
\eeqn
in an appropriate space of solutions,
and establish global well-posedness for the Cauchy problem
\begin{equation}\label{NLS}
i\partial_tQ=[-\Delta+\rho_Q,\Pi_\mu^-+Q]
\;\;,\;\; Q(0)=Q_0
\end{equation}
in two and three dimensions. As a crucial new ingredient that allow us to 
extend the work of 
Lewin-Sabin \cite{LS1} to the much more singular case
of the delta function potential, 
we establish new 
Strichartz estimates for density functions and density matrices  in Section \ref{sec-strichartz}. 
We remark that Lewin and Sabin used Strichartz-type estimates that they established in \cite{LS1}
in the case of positive temperature, but not in the zero temperature situation.
\\

\section{Statement of the main result} \label{sec-statement}

In this section, we first introduce some notation and relevant operator spaces.
Then, we present the main results and describe the
strategy of the proof. 

\subsection{Notation} 

For simplicity of exposition, we assume that the chemical potential has the value $\mu=1$. 
%Indeed, for $\mu>0$, one can easily modify just by replacing the reference state $\Pi_1^-$ by $\Pi_\mu^-$, the standard Japanese bracket $\la x\ra=\sqrt{1+|x|^2}$ by $\sqrt{\mu+|x|^2}$ and $(-\Delta-1)$ by $(-\Delta-\mu)$, respectively.

We will denote by $\mathfrak{S}^p$ the Schatten spaces
\begin{equation}
    \|Q\|_{\mathfrak{S}^p} := \Big({\rm Tr}|Q|^p \Big)^{1/p}
\end{equation}
for $p\geq1$.
We define the Banach space $\mathcal{X}$ by the collection of self-adjoint operators on $L^2$, equipped with the norm
\begin{equation}
\|Q\|_{\mathcal{X}}=\|Q\|_{\textup{Op}}+\sum_{\pm}\||\Delta+1|^{\frac{1}{2}}Q^{\pm\pm}|\Delta+1|^{\frac{1}{2}}\|_{\mathfrak{S}^1},
\end{equation}
where 
$Q^{\pm\pm}=\Pi_1^\pm Q\Pi_1^\pm$, 
$\Pi_1^-=\mathbf{1}_{(-\Delta\leq1)}$ 
and 
$\Pi_1^+=\mathbf{1}_{(-\Delta>1)}$.

For an operator $Q\in\mathcal{X}$, we denote 
\begin{equation}\label{relative kinetic energy}
\textup{Tr}_0(-\Delta-1)Q:=\sum_{\pm}\textup{Tr}|\Delta+1|^{\frac{1}{2}}Q^{\pm\pm}|\Delta+1|^{\frac{1}{2}},
\end{equation}
and we call it the \textit{relative kinetic energy} from the reference state $\Pi_1^-$. 
We note that  for a finite-rank smooth operator $Q$, 
$$
	\textup{Tr}_0(-\Delta-1)Q=\textup{Tr}(-\Delta-1)Q.
$$ 
For an operator $Q\in\mathcal{X}$, the relative kinetic energy can be expressed as the limit of $\textup{Tr}(-\Delta-1)Q_n$ as $n\to \infty$, where $\{Q_n\}_{n=1}^\infty$ is a sequence of finite-rank smooth operators such that $Q_n\to Q$ in $\mathcal{X}$ as $n\to \infty$ (see Lemma 3.2 in \cite{FLLS1}).

\subsubsection*{\bf The space of initial data} 
We define the \textit{relative energy space}, which will contain the initial data for our main global well-posedness result, as the collection of perturbations having finite operator norm and relative kinetic energy, 
\begin{equation}\label{relative energy space}
\mathcal{K}:=\Big\{Q=\gamma-\Pi^-\in\mathcal{X}:\ 0\leq\gamma\leq 1\Big\}.
\end{equation}
For an operator $Q\in\mathcal{K}$, we define the \textit{relative energy} of the NLS system by
\begin{equation}
\mathcal{E}(Q):=\textup{Tr}_0(-\Delta-1)Q+\frac{1}{2}\int_{\mathbb{R}^d}(\rho_Q)^2 dx.
\end{equation}

\begin{remark}
We note that if $Q\in\mathcal{K}$, the relative kinetic energy is positive definite, 
and is well-defined (see Lemma \ref{relative kinetic energy'}).
\end{remark}

\begin{remark}
In this article, we restrict ourselves to two and three-dimensions, since otherwise, the relative energy is not well-defined in the relative energy space $\mathcal{K}$, because the potential energy is not bounded by the relative kinetic energy in other dimensions. Indeed, in the proof of Lemma \ref{relative kinetic energy'}, the use of the Lieb-Thirring inequality \eqref{LT inequality} fails when $d=1$, and the use of the generalized Sobolev inequality \eqref{Sobolev inequality} fails when $d\geq 4$.
\end{remark}

\subsubsection*{\bf The solution space}
The solution to the NLS system \eqref{NLS} that we obtain in this paper belongs to the space
 $\mathfrak{Y}^1$ which is defined as follows. Given $I\subset\mathbb{R}$, we define the Banach space $\mathfrak{Y}^\alpha(I)$ of solutions by
\begin{equation}\label{Y alpha norm}
\|Q\|_{\mathfrak{Y}^\alpha(I)}:=\|Q\|_{C_t(I;\textup{Op})}+\|\Pi_2^+Q\|_{\mathcal{S}^\alpha(I)}+\|\rho_Q\|_{L_t^2(I;H^{\alpha+\frac{1}{2}-\eta})\bigcap L_t^\infty(I; L^2)},
\end{equation}
where $\mathcal{S}^\alpha(I)$ is a Strichartz space of operator kernels with $\alpha$ derivatives;
its precise definition is given in \eqref{eq-StrichNorm-def-1}, below.
Moreover,  $\eta=\eta(d,\alpha)\geq0$ is either small or zero, depending on $d$ and $\alpha$, and is 
defined in \eqref{eta}.

\subsection{The main result and a description of the proof}

We now state our main theorem. 
\begin{theorem}[Global well-posedness]\label{main theorem}
Let $d=2,3$. For initial data $Q_0\in\mathcal{K}$, there exist arbitrarily large
$T^-, T^+ > 0$ and a unique global solution $Q
\in\mathfrak{Y}^1((-T^-,T^+))$ to the system \eqref{NLS}. 
%For initial data $Q_0\in\mathcal{K}$, there exists a unique global solution $Q \in\mathfrak{Y}^1(\R)$ to the system \eqref{NLS}. 
Moreover, $Q(t)\in\mathcal{K}$ and the relative energy is conserved, i.e., 
$$
	\mathcal{E}(Q(t))=\mathcal{E}(Q_0) \mbox{ for all } t\in\mathbb{R}.
$$
\end{theorem}

%This result allows us to enhance local well-posedness to global well-posedness for \eqref{NLS}.

Our approach to proving Theorem \ref{main theorem} is motivated by the strategy that Lewin and Sabin used in \cite{LS1}.
However, in order to implement it in the case of the singular $\delta$ potential considered
in our paper, we introduce new Strichartz estimates for density functions, and
consequently need to employ
different intermediate solutions spaces for our analysis.
Similar arguments are used for the proof of energy conservation for solutions
to dispersive nonlinear PDE in the energy
space, see e.g. \cite{Caz}.

We will now summarize the key steps in our construction. 

\bigskip

{\bf Step 1}: In  Section \ref{sec-XYspaceLWP} we prove local well-posedness
of the NLS system \eqref{NLS} in a space containing ${\mathcal K}$. 

More precisely, we define the Banach space of initial data, $\mathfrak{X}^\alpha$, for $\alpha\geq1$, by
\begin{equation}
\|Q\|_{\mathfrak{X}^\alpha}:=\|Q\|_{\textup{Op}}+\|\Pi_2^+Q\|_{\mathcal{H}^\alpha}\,,
\end{equation}
where $\mathcal{H}^\alpha$ is the Hilbert-Schmidt type Sobolev space defined as follows: 
\begin{equation} \label{def-HSSob}
\|Q\|_{\mathcal{H}^\alpha}:=\|\la\nabla\ra^\alpha Q\la\nabla\ra^\alpha\|_{\mathfrak{S}^2}=\|\la\nabla_x\ra^\alpha\la\nabla_{x'}\ra^\alpha Q(x,x')\|_{L_x^2L_{x'}^2}.
\end{equation}
We note that 
\begin{equation}
    {\mathcal K}\subset\mathfrak{X}^\alpha
\end{equation}
for any $\alpha\geq1$.

Then our local well-posedness results can be stated as follows: 

\begin{theorem}[Local well-posedness]\label{LWP0}
Let $d=2,3$ and $\alpha\geq 1$. Given initial data $Q_0\in\mathfrak{X}^\alpha\cap\mathcal{K}$, there exists an interval $I$ and a unique solution $Q\in\mathfrak{Y}^\alpha(I)$ to the equation \eqref{NLS}.
\end{theorem}

Crucial ingredients in the proof of this local well-posedness result are new Strichartz estimates for the density 
function $\rho$ which we prove in Section \ref{sec-strichartz} and in Proposition \ref{local-in-time Strichartz}.

\bigskip

{\bf Step 2:} In Section \ref{Conservation of the relative energy at 0T}, we prove that 
the solutions constructed in Theorem \ref{LWP0} preserve the relative energy. More precisely, we prove the following: 

\begin{theorem}[Conservation of the relative energy]\label{relative energy conservation0}
Let $d=2,3$. Let $Q(t)\in\mathfrak{Y}^1(I)$ be the solution to the equation \eqref{NLS} with initial data $Q_0\in\mathcal{K}$ (constructed in Theorem \ref{LWP0}). Then,
\begin{equation}
\mathcal{E}(Q(t))=\mathcal{E}(Q_0),\quad\forall t\in I.
\end{equation}
\end{theorem}

\noindent
The proof of Theorem \ref{relative energy conservation0} can be summarized as follows: 

\begin{enumerate} 

\item  To begin with, we approximate the initial data $Q_0\in\mathcal{K}$ 
by a sequence $Q^{(n)}\in \mathfrak{H}^2$ of very regular kernels, where
\begin{equation}\label{trace H2}
\|Q^{(n)}\|_{\mathfrak{H}^2}:=\|\la\nabla\ra^2Q^{(n)}\la\nabla\ra^2\|_{\mathfrak{S}^1}.
\end{equation}

\vspace{0.2in}
 
\item The local well-posedness for such regular solutions is established in Section \ref{sec-H2encons}.  
For the precise statement see Theorem \ref{trace LWP},
which gives the existence of a
unique solution $Q\in C_t(I;\mathfrak{H}^2)$ to the equation \eqref{NLS}  for initial data $Q_0\in\mathfrak{H}^2$.

\vspace{0.2in}

\item Owing to their $\mathfrak{H}^2$ regularity, energy conservation 
for such solutions can be proved in a straightforward manner, see Section \ref{sec-H2encons}. Hence, one obtains 
that every $Q(t)\in C_t(I;\mathfrak{H}^2)$ which is a solution to the equation \eqref{NLS} for initial data $Q_0$,
\begin{equation}
\mathcal{E}(Q(t))=\mathcal{E}(Q_0),\quad\forall t\in I
\end{equation}
is satisfied.
This is formulated precisely in  Proposition \ref{conservation law for regular solutions}.

\vspace{0.2in}

\item The last step in the proof of Theorem \ref{relative energy conservation0}, which is carried out in 
Section \ref{Conservation of the relative energy at 0T}, consists of showing that 
\begin{equation}
\begin{aligned}
\mathcal{E}(Q(t))&\leq\liminf_{n\to\infty}\mathcal{E}(Q^{(n)}(t))
=\liminf_{n\to\infty}\mathcal{E}(Q^{(n)}(0)) =\mathcal{E}(Q_0).
\end{aligned}
\end{equation}
Repeating the same argument backwards in time, we obtain $\mathcal{E}(Q_0)\leq \mathcal{E}(Q(t))$. 
Thus, we conclude that $\mathcal{E}(Q_0)= \mathcal{E}(Q(t))$.

  \end{enumerate} 

{\bf Step 3:} Finally, in Section \ref{sec-loctoglob}
we conclude the proof of Theorem \ref{main theorem}
by showing that the local solution constructed in Theorem \ref{LWP0}, which satisfies conservation of relative energy, 
can be extended for all times.

 \section{Strichartz estimates}  \label{sec-strichartz}

In this section, we present the main tools to prove local well-posedness of the NLS system  \eqref{NLS} are new Strichartz estimates for operator kernels and 
for density functions. 
The latter are obtained in this paper using space-time Fourier transform techniques, which were instrumental in obtaining bilinear Strichartz estimates
in the context of dispersive \cite{B1, B2} and wave \cite{KM1} equations, and
more recently,
in the context of the Gross-Pitaevskii hierarchy \cite{KM2, CP2, Xie} (which is an infinite hierarchy 
of coupled linear PDEs that describes the dynamics of infinitely many bosons; it appears in the derivation of the nonlinear Schr\"{o}dinger equation from quantum many body systems). 
 
We note that recently, in \cite{FLLS2}, Frank, Lewin, Lieb and Seiringer established Strichartz estimates for density functions of the form
\begin{equation}\label{Strichartz with no derivative}
\|\rho_{e^{it\Delta}\gamma_0e^{-it\Delta}}\|_{L_t^p(\mathbb{R}; L^q)}\lesssim \|\gamma_0\|_{\mathfrak{S}^{\frac{2q}{q+1}}},
\end{equation}
where $p,q,d\geq 1$, $1\leq q\leq\frac{d+2}{d}$ and $\frac{2}{p}+\frac{d}{q}=d$. In particular, when $\gamma_0=\sum_{j=1}^N|\phi_j\ra\la\phi_j|$ is a projection operator with an orthonormal set $\{\phi_j\}_{j=1}^N$ in $L^2$, the inequality \eqref{Strichartz with no derivative} can be read as
\begin{equation}\label{Strichartz with no derivative, example}
\Big\|\sum_{j=1}^N|e^{it\Delta}\phi_j|^2\Big\|_{L_t^p(\mathbb{R}; L^q)}\lesssim N^{\frac{q+1}{2q}}.
\end{equation}
An important feature of \eqref{Strichartz with no derivative, example} is that it improves summability compared to the trivial consequence of Strichartz estimates for the Schr\"odinger flow $e^{it\Delta}$, 
\begin{equation}
\Big\|\sum_{j=1}^N|e^{it\Delta}\phi_j|^2\Big\|_{L_t^p(\mathbb{R}; L^q)}\leq\sum_{j=1}^N\|e^{it\Delta}\phi_j\|_{L_t^{2p}(\mathbb{R}; L^{2q})}^2\lesssim\sum_{j=1}^N\|\phi_j\|_{L_x^2}^2= N.
\end{equation}
Indeed, the exponent $\frac{q+1}{2q}$ on the right hand side of \eqref{Strichartz with no derivative, example} is strictly less than one unless $q=1$. In this sense, the Strichartz inequality \eqref{Strichartz with no derivative} is a generalization of the kinetic energy inequality
\begin{equation}\label{Lieb-Thirring}
\int_{\mathbb{R}^d}\Big\{\sum_{j=1}^N|\phi_j(x)|^2\Big\}^{1+\frac{2}{d}}dx\lesssim \int_{\mathbb{R}^d}\sum_{j=1}^N|\nabla \phi_j(x)|^2dx,
\end{equation}
which is dual to the famous Lieb-Thirring inequality \cite{LT1, LT2}. Later, in \cite{FS}, Frank and Sabin extended \eqref{Strichartz with no derivative} to the optimal range of $q$, that is, $1\leq q<\frac{d+1}{d-1}$. In  \cite{LS1, LS2}, Lewin and Sabin employed these Strichartz estimates in their study on the Hartree equation for infinitely many fermions.

In this section, we introduce a different kind of Strichartz estimates for density functions of the form
\begin{equation}\label{Strichartz with derivative}
\||\nabla|^{\frac{1}{2}}\rho_{e^{it\Delta}\gamma_0e^{-it\Delta}}\|_{L_t^2(\mathbb{R};H^{\alpha_1})}\lesssim\|\la\nabla\ra^\alpha\gamma_0\la\nabla\ra^\alpha\|_{\mathfrak{S}^2}
\end{equation}
(see Theorem \ref{Strichartz estimates for density functions}). We remark that compared to the Strichartz estimates \eqref{Strichartz with no derivative}, there is more gain in summability, equivalently a larger Schatten exponent ($2>\frac{2q}{q+1}$) on the right hand side, assuming more regularity on operators. Moreover, there is an improvement in regularity on the density function $\rho_{e^{it\Delta}\gamma_0e^{-it\Delta}}$. Indeed, $\alpha_1+\frac{1}{2}>\alpha$ in Theorem \ref{Strichartz estimates for density functions}. This fact will play a crucial role in proving our main theorem (see Lemma 5.4). The proof of Strichartz estimates \eqref{Strichartz with derivative} is based on the space-time Fourier transform, and is inspired by the proof of bilinear Strichartz estimates in Klainerman and Machedon \cite{KM1, KM2}, Bourgain \cite{B1, B2} , Chen and Pavlovi\'{c} \cite{CP2} and Xie \cite{Xie}. 

Another new ingredient in this article is Strichartz estimates for operator kernels (see Theorem \ref{Strichartz estimates for operator kernels}), which enjoy the smoothing property of the Schr\"odinger flow $e^{it\Delta}\gamma_0e^{-it\Delta}$ from the Hilbert-Schmidt operators. Although they are quite natural for dispersive equations in the Heisenberg picture, to the best of the authors' knowledge, it is the first time that this kind of Strichartz estimates appear in the literature.

\subsection{Strichartz estimates for operator kernels}
For $\alpha\geq 0$, we define the \textit{Hilbert-Schmidt Sobolev space} $\mathcal{H}^\alpha$ by the collection of Hilbert-Schmidt operators, which are not necessarily self-adjoint, with the norm
\begin{equation}
\|\gamma_0\|_{\mathcal{H}^\alpha}:=\|\la\nabla\ra^\alpha\gamma_0\la\nabla\ra^\alpha\|_{\mathfrak{S}^2}=\|\la\nabla_x\ra^\alpha\la\nabla_{x'}\ra^\alpha \gamma_0(x,x')\|_{L_x^2L_{x'}^2},
\end{equation}
where $\gamma_0(x,x')$ is the integral kernel of $\gamma_0$, i.e.,  
\begin{equation}
(\gamma_0 g)(x)=\int_{\mathbb{R}^d}\gamma_0(x,x')g(x')dx'.
\end{equation}
We say that an exponent pair $(q,r)$ is \textit{admissible} if $2\leq q,r\leq\infty$, $(q,r,d)\neq (2,\infty,2)$ and 
\begin{equation}
\frac{2}{q}+\frac{d}{r}=\frac{d}{2}.
\end{equation}
For a time-dependent operator $\gamma(t)$ on an interval $I\subset\mathbb{R}$, we define its Strichartz norm by
\begin{equation}
\begin{aligned}\label{eq-StrichNorm-def-1}
\|\gamma(t)\|_{\mathcal{S}^\alpha(I)}:=\sup_{(q,r)\textup{: admissible}}&
\Big\{\|\la\nabla_x\ra^\alpha\la\nabla_{x'}\ra^\alpha\gamma(t,x,x')\|_{L_{t}^q(I; L_x^rL_{x'}^2)}\\
&+\|\la\nabla_x\ra^\alpha\la\nabla_{x'}\ra^\alpha\gamma(t,x,x')\|_{L_{t}^q (I;L_{x'}^rL_{x}^2)}\Big\}.
\end{aligned}
\end{equation}
It is obvious that $\mathcal{S}^\alpha(I) \hookrightarrow L_t^\infty(I;\mathcal{H}^\alpha)$.

We identify the operator $e^{it\Delta}\gamma_0 e^{-it\Delta}$ with its integral kernel
\begin{equation}
(e^{it\Delta}\gamma_0 e^{-it\Delta})(x,x')=(e^{it(\Delta_x-\Delta_{x'})}\gamma_0)(x,x').
\end{equation}
Then, as a function, the dispersive estimate for the linear propagator, one of whose spatial variables is frozen, yields the following Strichartz estimates.
 
\begin{theorem}[Strichartz estimates for operator kernels]\label{Strichartz estimates for operator kernels}
Let $I\subset\mathbb{R}$. Then, we have
\begin{equation}
\begin{aligned}
\|e^{it\Delta}\gamma_0e^{-it\Delta}\|_{\mathcal{S}^\alpha(\mathbb{R})}&\lesssim\|\gamma_0\|_{\mathcal{H}^\alpha},\\
\Big\|\int_0^t e^{i(t-s)\Delta}R(s)e^{-i(t-s)\Delta}ds\Big\|_{\mathcal{S}^\alpha(\mathbb{R})}&\lesssim\|R(t)\|_{L_t^1(\mathbb{R};\mathcal{H}^\alpha)}.
\end{aligned}
\end{equation}
\end{theorem}

\begin{proof}
By symmetry, it suffices to show that 
\begin{equation}\label{Strichartz estimates for operator kernels proof}
\begin{aligned}
\|e^{it(\Delta_x-\Delta_{x'})}\gamma_0\|_{L_{t}^q (\mathbb{R};L_x^rL_{x'}^2)}&\lesssim\|\gamma_0\|_{L_x^2L_{x'}^2},\\
\Big\|\int_0^t e^{i(t-s)(\Delta_x-\Delta_{x'})}R(s)ds\Big\|_{L_{t}^q(\mathbb{R}; L_x^rL_{x'}^2)}&\lesssim\|R(t)\|_{L_{t}^1(\mathbb{R}; L_x^2L_{x'}^2)}.
\end{aligned}
\end{equation}
We denote by $L^p(L^2)$ the collection of single variable $L^p$-functions $f(x,\cdot): \mathbb{R}^d\to L^2$, which is identified with the collection of two spatial variable $L_x^pL_{x'}^2$-functions. It is obvious that by unitarity, 
\begin{equation}
\|e^{it(\Delta_x-\Delta_{x'})}\gamma_0\|_{L^2(L^2)}=\|e^{it(\Delta_x-\Delta_{x'})}\gamma_0\|_{L_x^2L_{x'}^2}=\|\gamma_0\|_{L_x^2L_{x'}^2}=\|\gamma_0\|_{L^2(L^2)}.
\end{equation}
On the other hand, by unitarity of the linear propagator $e^{-it\Delta_{x'}}$, the dispersive estimate $\|e^{it\Delta_x}\|_{L_x^1\to L_{x}^\infty}\lesssim|t|^{-d/2}$ and the Minkowski inequality, we get
\begin{equation}
\begin{aligned}
\|e^{it(\Delta_x-\Delta_{x'})}\gamma_0\|_{L^\infty (L^2)}&=\|e^{it(\Delta_x-\Delta_{x'})}\gamma_0\|_{L_x^\infty L_{x'}^2}=\|e^{it\Delta_x}\gamma_0\|_{L_x^\infty L_{x'}^2}\\
&\leq \|e^{it\Delta_x}\gamma_0\|_{L_{x'}^2L_x^\infty}\lesssim|t|^{-d/2}\|\gamma_0\|_{L_{x'}^2L_x^1}\\
&\leq|t|^{-d/2}\|\gamma_0\|_{L_{x}^1L_{x'}^2}=|t|^{-d/2}\|\gamma_0\|_{L^1(L^2)}.
\end{aligned}
\end{equation}
Then, \eqref{Strichartz estimates for operator kernels proof} follows from the abstract version of Strichartz estimates (Theorem 10.1 in Keel and Tao \cite{KT}) with $B_0=L^2(L^2)$, $B_1^*=L^\infty (L^2)$, $H=L^2(L^2)$ and $\sigma=\frac{d}{2}$.
\end{proof}

\begin{remark}
The proof of Theorem \ref{Strichartz estimates for operator kernels} relies on the fact that an operator $\gamma$ in $\mathfrak{S}^2$ can be identified with its kernel $\gamma(x,x')$ as a function in $L_{x}^2 L_{x'}^2$. An interesting open question is to derive similar Strichartz estimates for operators in different Schatten classes.
\end{remark}

\subsection{Strichartz estimates for density functions}
Next, we establish the Strichartz estimates for density functions $\rho_{e^{it\Delta}\gamma_0e^{-it\Delta}}$. 
\begin{theorem}[Strichartz estimates for density functions]\label{Strichartz estimates for density functions}
Suppose that
\begin{equation}\label{alpha}
\left\{\begin{aligned}
&\alpha\geq0&&\text{ when }d=1,\\
&\alpha>\tfrac{d-1}{4}&&\text{ when }d\geq2,
\end{aligned}
\right.
\end{equation}
and
\begin{equation}\label{alpha_1}
\left\{\begin{aligned}
&\alpha_1=\alpha&&\text{ when }d=1,\\
&\alpha_1=2\alpha-\tfrac{d-1}{2}&&\text{ when }d\geq 2\textup{ and }\tfrac{d-1}{4}<\alpha<\tfrac{d-1}{2},\\
&\alpha_1<\tfrac{d-1}{2}&&\text{ when }d\geq 2\textup{ and }\alpha=\tfrac{d-1}{2},\\
&\alpha_1=\alpha&&\text{ when }d\geq 2\textup{ and }\alpha>\tfrac{d-1}{2}.
\end{aligned}
\right.
\end{equation}
$(i)$ (Homogeneous Strichartz estimate)
\begin{equation}\label{homogeneous Strichartz estimates for density functions}
\||\nabla|^{\frac{1}{2}}\rho_{e^{it\Delta}\gamma_0e^{-it\Delta}}\|_{L_t^2(\mathbb{R};H^{\alpha_1})}\lesssim\|\gamma_0\|_{\mathcal{H}^\alpha}.
\end{equation}
$(ii)$ (Inhomogeneous Strichartz estimate)
\begin{equation}\label{inhomogeneous Strichartz estimates for density functions}
\Big\||\nabla|^{\frac{1}{2}}\rho\Big[ \int_0^t e^{i(t-s)\Delta}R(s) e^{-i(t-s)\Delta}ds \Big]\Big\|_{L_t^2(\mathbb{R};H^{\alpha_1})}\lesssim\|R(t)\|_{L_t^1(\mathbb{R};\mathcal{H}^\alpha)}.
\end{equation}
\end{theorem}

\begin{proof}
$(i)$: We prove $(i)$ by duality. The advantage of considering a dual inequality is that one can prove optimality with a small modification (see Proposition \ref{optimality} and its proof).

We write the space-time Fourier transform of $\rho_{e^{it\Delta}\gamma_0e^{-it\Delta}}$,
\begin{equation}
\begin{aligned}
(\rho_{e^{it\Delta}\gamma_0e^{-it\Delta}})^{\sim}(\tau,\xi)&=\mathcal{F}_{t,x}\Big\{\frac{1}{(2\pi)^{2d}}\int_{\mathbb{R}^{2d}} e^{-it(|\xi_1|^2-|\xi_2|^2)}\hat{\gamma}_0(\xi_1,\xi_2) e^{ix\cdot(\xi_1+\xi_2)} d\xi_1 d\xi_2\Big\}\\
&=\mathcal{F}_{t}\Big\{\frac{1}{(2\pi)^{d}}\int_{\mathbb{R}^{2d}} e^{-it(|\xi_1|^2-|\xi_2|^2)}\hat{\gamma}_0(\xi_1,\xi_2) \delta(\xi-\xi_1-\xi_2) d\xi_1 d\xi_2\Big\}\\
&=\mathcal{F}_{t}\Big\{\frac{1}{(2\pi)^{d}}\int_{\mathbb{R}^{d}} e^{-it(|\xi_1|^2-|\xi-\xi_1|^2)}\hat{\gamma}_0(\xi_1,\xi-\xi_1) d\xi_1\Big\}\\
&=\frac{1}{(2\pi)^{d-1}}\int_{\mathbb{R}^{d}} \delta(\tau+|\xi_1|^2-|\xi-\xi_1|^2)\hat{\gamma}_0(\xi_1,\xi-\xi_1) d\xi_1.
\end{aligned}
\end{equation}
Then, by the Plancherel theorem,
\begin{equation}
\begin{aligned}
&\int_{\mathbb{R}^{d+1}} (\la\nabla\ra^{\alpha_1}|\nabla|^{\frac{1}{2}}\rho_{e^{it\Delta}\gamma_0e^{-it\Delta}})(x)\overline{V(t,x)} dx dt\\
&=\int_{\mathbb{R}^{d+1}} \big(\la\nabla\ra^{\alpha_1}|\nabla|^{\frac{1}{2}}\rho_{e^{it\Delta}\gamma_0e^{-it\Delta}}\big)^\sim(\tau,\xi)\overline{\tilde{V}(\tau,\xi)} d\xi d\tau\\
&=\frac{1}{(2\pi)^{d-1}}\int_{\mathbb{R}^{2d+1}} \la\xi\ra^{\alpha_1} |\xi|^{\frac{1}{2}}\delta(\tau+|\xi_1|^2-|\xi-\xi_1|^2)\hat{\gamma}_0(\xi_1,\xi-\xi_1) \overline{\tilde{V}(\tau,\xi)} d\xi_1d\xi d\tau\\
&=\frac{1}{(2\pi)^{d-1}}\int_{\mathbb{R}^{2d}} \la\xi\ra^{\alpha_1}|\xi|^{\frac{1}{2}}\hat{\gamma}_0(\xi_1,\xi-\xi_1) \overline{\tilde{V}(-|\xi_1|^2+|\xi-\xi_1|^2,\xi)} d\xi_1d\xi,
\end{aligned}
\end{equation}
where $\tau$ is integrated out in the last identity. Therefore, by duality, \eqref{homogeneous Strichartz estimates for density functions} is equivalent to the inequality 
\begin{equation}\label{dual inequality}
\Big\|\frac{\la\xi\ra^{\alpha_1}|\xi|^{\frac{1}{2}}\tilde{V}(-|\xi_1|^2+|\xi-\xi_1|^2,\xi)}{\la\xi_1\ra^{\alpha}\la\xi-\xi_1\ra^{\alpha}}\Big\|_{L_{\xi,\xi_1}^2}\lesssim \|V\|_{L_{t\in\mathbb{R}}^2 L_x^2}.
\end{equation}

We consider the square of the left hand side of \eqref{dual inequality}, 
\begin{equation}\label{I}
I=\int_{\mathbb{R}^d}\Big\{\int_{\mathbb{R}^d}\frac{\la\xi\ra^{2\alpha_1}|\xi||\tilde{V}(-2\xi\cdot\xi_1+|\xi|^2,\xi)|^2}{\la\xi_1\ra^{2\alpha}\la\xi-\xi_1\ra^{2\alpha}} d\xi_1\Big\}d\xi.
\end{equation}
In the one-dimensional case, we assume that $\alpha_1=\alpha$. Then, $\frac{\la\xi\ra^{2\alpha}}{\la\xi_1\ra^{2\alpha}\la \xi-\xi_1\ra^{2\alpha}}$ is bounded, since either $|\xi_1|$ or $|\xi-\xi_1|$ is greater than or equal to $\frac{|\xi|}{2}$. Thus, changing variables $\tau=-2\xi\xi_1+\xi^2$, we get
\begin{equation}
\begin{aligned}
I&\lesssim\int_{\mathbb{R}}\Big\{\int_{\mathbb{R}}|\xi||\tilde{V}(-2\xi\xi_1+\xi^2,\xi)|^2d\xi_1\Big\}d\xi\\
&=\frac{1}{2}\int_{\mathbb{R}^2}|\tilde{V}(\tau,\xi)|^2d\tau d\xi=\frac{1}{2}\|V\|_{L_{t\in\mathbb{R}}^2L_x^2}^2.
\end{aligned}
\end{equation}
Suppose that $d\geq 2$. Given $\xi\in\mathbb{R}^d$, we introduce new variables $q=q(\xi)=(q_1,\cdots, q_d)\in\mathbb{R}^d$ such that $\xi_1=q_1e_1+\cdots+q_de_d$, where $\{e_j\}_{j=1}^d$ is an orthonormal basis in $\mathbb{R}^d$ with $e_1=\frac{\xi}{|\xi|}$. Then, changing variables by $\xi_1\mapsto q$, we write
\begin{equation}
I=\int_{\mathbb{R}^{2d-1}}\Big\{\int_{\mathbb{R}}\frac{\la\xi\ra^{2\alpha_1}|\xi||\tilde{V}(-2|\xi|q_1+|\xi|^2,\xi)|^2}{\la q\ra^{2\alpha}\la(q_1-|\xi|,q')\ra^{2\alpha}} dq_1\Big\} dq'd\xi,
\end{equation}
where $q'=(q_2,\cdots, q_d)$ is a vector in $\mathbb{R}^{d-1}$. Next, we change variables $\tau=-2|\xi|q_1+|\xi|^2$, and denote $q_1^*=q_1(\tau,\xi)=-\frac{\tau-|\xi|^2}{2|\xi|}$. Then, 
\begin{equation}\label{I'}
\begin{aligned}
I&=\int_{\mathbb{R}^{2d-1}}\Big\{\int_{\mathbb{R}}\frac{\la\xi\ra^{2\alpha_1}|\tilde{V}(\tau,\xi)|^2}{2\la (q_1^*, q')\ra^{2\alpha}\la(q_1^*-|\xi|,q')\ra^{2\alpha}} d\tau\Big\} dq'd\xi\\
&=\int_{\mathbb{R}^{d+1}}\Big\{\int_{\mathbb{R}^{d-1}}\frac{\la\xi\ra^{2\alpha_1}dq'}{2\la (q_1^*,q')\ra^{2\alpha}\la(q_1^*-|\xi|,q')\ra^{2\alpha}} \Big\}|\tilde{V}(\tau,\xi)|^2d\xi d\tau.
\end{aligned}
\end{equation}
Thus, in order to prove \eqref{dual inequality}, it suffices to show that the integral $\{\cdots\}$ on the second line of \eqref{I'} is bounded uniformly in $\tau$ and $\xi$.

Suppose that $|q_1^*|\leq\frac{|\xi|}{2}$ $(\Rightarrow |q_1^*-|\xi||\geq\frac{|\xi|}{2})$. Then, 
\begin{equation}\label{integral 1}
\begin{aligned}
&\int_{|q'|\leq2|\xi|}\frac{\la\xi\ra^{2\alpha_1}dq'}{2\la (q_1^*,q')\ra^{2\alpha}\la(q_1^*-|\xi|,q')\ra^{2\alpha}}\\
&\leq\int_{|q'|\leq2|\xi|}\frac{\la\xi\ra^{2\alpha_1}dq'}{2\la q'\ra^{2\alpha}\la q_1^*-|\xi|\ra^{2\alpha}}\\
&\leq\int_{|q'|\leq2|\xi|}\frac{\la\xi\ra^{2\alpha_1}}{\la q'\ra^{2\alpha}\la \xi\ra^{2\alpha}}dq'=\int_{|q'|\leq2|\xi|}\frac{\la\xi\ra^{2\alpha_1-2\alpha}}{\la q'\ra^{2\alpha}}dq'\\
&\lesssim\left\{\begin{aligned}
&\la\xi\ra^{2\alpha_1-4\alpha+(d-1)}&&\text{ when }\alpha<\tfrac{d-1}{2},\\
&\la\xi\ra^{2\alpha_1-(d-1)}\la \ln|\xi|\ra&&\text{ when }\alpha=\tfrac{d-1}{2},\\
&\la\xi\ra^{2\alpha_1-2\alpha}&&\text{ when }\alpha>\tfrac{d-1}{2}.
\end{aligned}
\right.
\end{aligned}
\end{equation}
On the other hand, if $|q_1^*|\geq\frac{|\xi|}{2}$, then 
\begin{equation}\label{integral 2}
\begin{aligned}
&\int_{|q'|\leq2|\xi|}\frac{\la\xi\ra^{2\alpha_1}dq'}{2\la (q_1^*,q')\ra^{2\alpha}\la(q_1^*-|\xi|,q')\ra^{2\alpha}}\\
&\leq\int_{|q'|\leq2|\xi|}\frac{\la\xi\ra^{2\alpha_1}dq'}{2\la q_1^*\ra^{2\alpha}\la q'\ra^{2\alpha}}\\
&\leq\int_{|q'|\leq2|\xi|}\frac{\la\xi\ra^{2\alpha_1}}{\la \xi\ra^{2\alpha}\la q'\ra^{2\alpha}}dq'=\int_{|q'|\leq2|\xi|}\frac{\la\xi\ra^{2\alpha_1-2\alpha}}{\la q'\ra^{2\alpha}}dq'\\
&\lesssim\left\{\begin{aligned}
&\la\xi\ra^{2\alpha_1-4\alpha+(d-1)}&&\text{ when }\alpha<\tfrac{d-1}{2},\\
&\la\xi\ra^{2\alpha_1-(d-1)}\la \ln|\xi|\ra&&\text{ when }\alpha=\tfrac{d-1}{2},\\
&\la\xi\ra^{2\alpha_1-2\alpha}&&\text{ when }\alpha>\tfrac{d-1}{2}.
\end{aligned}
\right.\end{aligned}
\end{equation}
For \eqref{integral 1}, \eqref{integral 2} to be uniformly bounded  in $\tau$ and $\xi$, 
the exponent $\alpha_1$ in \eqref{alpha_1} must be chosen so that 
\begin{equation}\label{eq-alph1-conds-1}
\begin{aligned}
\left\{\begin{aligned}
&2\alpha_1-4\alpha+(d-1)=0&&\text{ when }\alpha<\tfrac{d-1}{2},\\
&2\alpha_1-(d-1)<0&&\text{ when }\alpha=\tfrac{d-1}{2},\\
&2\alpha_1-2\alpha=0&&\text{ when }\alpha>\tfrac{d-1}{2}.
\end{aligned}
\right.\end{aligned}
\end{equation} 

It remains to estimate the integral $\{\cdots\}$ in \eqref{I'} whose integral domain is restricted to $\{q'\in\mathbb{R}^{d-1}: |q'|\geq 2|\xi|\}$. Now, using that both $|(q_1^*,q')|$ and $|(q_1^*-|\xi|,q')|$ are greater than equal to $|q'|$, we prove that 
\begin{equation}\label{integral 3}
\int_{|q'|\geq2|\xi|}\frac{\la\xi\ra^{2\alpha_1}dq'}{2\la (q_1^*,q')\ra^{2\alpha}\la(q_1^*-|\xi|,q')\ra^{2\alpha}}\leq\int_{|q'|\geq2|\xi|}\frac{\la\xi\ra^{2\alpha_1}}{2\la q'\ra^{4\alpha}}dq'\lesssim \la\xi\ra^{2\alpha_1-4\alpha+(d-1)},
\end{equation}
where in the second inequality, we used the assumption $\alpha>\frac{d-1}{4}$. 
Furthermore, for all of the cases in \eqref{eq-alph1-conds-1}, we have that the exponent on the r.h.s. in \eqref{integral 3} satisfies
\begin{equation}
     2\alpha_1-4\alpha+(d-1)\leq0
\end{equation} 
if concurrently, $\alpha>\frac{d-1}{4}$ holds.
Therefore, we conclude that the above integrals, \eqref{integral 1}, \eqref{integral 2} and \eqref{integral 3}, are bounded uniformly in $\tau$ and $\xi$.

This establishes \eqref{dual inequality} which is equivalent 
to \eqref{homogeneous Strichartz estimates for density functions}.

$(ii)$: By the Minkowski inequality and $(i)$, we prove that
\begin{equation}
\begin{aligned}
&\Big\||\nabla|^{\frac{1}{2}}\rho\Big[ \int_0^t e^{i(t-s)\Delta}R(s) e^{-i(t-s)\Delta}ds \Big]\Big\|_{L_t^2(\mathbb{R};H^{\alpha_1})}\\
&\leq\int_\mathbb{R} \||\nabla|^{\frac{1}{2}}\rho_{e^{it\Delta}(e^{-is\Delta}R(s)e^{is\Delta}) e^{-it\Delta}}\|_{L_t^2(\mathbb{R}_t;H^{\alpha_1})}ds\\
&\lesssim\int_\mathbb{R} \|e^{-is\Delta}R(s)e^{is\Delta}\|_{\mathcal{H}^\alpha}ds=\int_\mathbb{R} \|R(s)\|_{\mathcal{H}^\alpha}ds=\|R(t)\|_{L_t^1(\mathbb{R};\mathcal{H}^\alpha)}.
\end{aligned}
\end{equation}
This completes the proof.
\end{proof}

Now, we may prove optimality of the homogeneous Strichartz estimate (Theorem \ref{Strichartz estimates for density functions} $(i)$). Precisely, we will show that \eqref{homogeneous Strichartz estimates for density functions} fails in the case $d\geq 2$ and $\alpha=\frac{d-1}{4}$, or the case $d\geq 2$ and $\alpha=\alpha_1=\frac{d-1}{4}$. Even more than that, we will show that in the first case, it is possible that $\rho_{e^{it\Delta}\gamma_0e^{-it\Delta}}$ is not even a distribution (see \eqref{eq: Strichartz-}). %%We remark that proving failure of the inequality in a very weak norm is not just for a mathematical curiosity. Indeed, if one tries to prove global well-posedness of the three-dimensional Hartree equation of infinitely many particles with a sufficiently nice convoluted potential $w$ at positive temperature ``via a straightforward contraction mapping argument" as in later sections, one needs a Strichartz estimate of the form
%%\begin{equation}\label{Hartree failure}
%%\||\nabla|^{\frac{1}{2}}(w*\rho_{e^{it\Delta}\gamma_0e^{-it\Delta}})\|_{X}\lesssim\|\gamma_0\|_{\mathcal{H}^{\frac{1}{2}}}
%%\end{equation}
%%with a suitable norm $\|\cdot\|_X$. Here, the $\mathcal{H}^{\frac{1}{2}}$-norm of $\gamma_0$ is taken on the right hand side of \eqref{Hartree failure}, since so far, it is known that the relative entropy $\mathcal{H}(\gamma,\gamma_f)$ controls just the $\mathcal{H}^{\frac{1}{2}}$-norm of $\gamma-\gamma_f$. However, Proposition \ref{optimality} tells us that the inequality \eqref{Hartree failure} is not true, and thus it is not easy to obtain global well-posedness by the argument in this article.
We denote by $(C_0^\infty)'$ the space of distributions, defined on $\mathbb{R}^{d+1}=\mathbb{R}_t\times\mathbb{R}_x^d$, that is dual to the space $C_0^\infty=C_0^\infty(\mathbb{R}^{d+1})$ of compactly supported smooth functions.

\begin{proposition}[Optimality of the inequality \eqref{homogeneous Strichartz estimates for density functions}]\label{optimality}
Let $d\geq 2$. Then,
\begin{equation}\label{eq: Strichartz-}
\sup_{\|\gamma_0\|_{\mathcal{H}^{\frac{d-1}{4}}}=1}
\frac{\big\||\nabla_x|^{\frac{1}{2}}\rho_{e^{it\Delta}\gamma_0e^{-it\Delta}}\big\|_{(C_0^\infty)'}}{\|\gamma_0\|_{\mathcal{H}^{\frac{d-1}{4}}}}=\infty
\end{equation}
and 
\begin{equation}\label{eq: Strichartz- 2}
\sup_{\|\gamma_0\|_{\mathcal{H}^{\frac{d-1}{2}}}=1}
\frac{\big\||\nabla_x|^{\frac{1}{2}}\rho_{e^{it\Delta}\gamma_0e^{-it\Delta}}\big\|_{L_t^2(\mathbb{R};\mathcal{H}^{\frac{d-1}{2}})}}{\|\gamma_0\|_{\mathcal{H}^{\frac{d-1}{2}}}}=\infty.
\end{equation}
\end{proposition}

\begin{proof}
\eqref{eq: Strichartz-}: By dualization as in Theorem \ref{Strichartz estimates for density functions}, it is enough to show that the integral $I$ (see \eqref{I}) is unbounded for $V\in C_0^\infty(\mathbb{R}^{d+1})$. Repeating \eqref{I'}, we write 
\begin{equation}
I=\cdots=\int_{\mathbb{R}^{d+1}}\Big\{\int_{\mathbb{R}^{d-1}}\frac{dq'}{2\la (q_1^*, q')\ra^{\frac{d-1}{2}}\la(q_1^*-|\xi|,q')\ra^{\frac{d-1}{2}}} \Big\} |\tilde{V}(\tau,\xi)|^2d\xi d\tau.
\end{equation}
However, since the function
\begin{equation}
\frac{1}{\la (q_1^*, q')\ra^{\frac{d-1}{2}}\la(q_1^*-|\xi|,q')\ra^{\frac{d-1}{2}}}
\end{equation}
is not integrable over $\mathbb{R}^{d-1}$, we conclude that $I=\infty$.

\eqref{eq: Strichartz- 2}: By duality again, it suffices to find a bounded sequence $\{V_n(t,x)\}_{n=1}^\infty$ in $L_t^2(\mathbb{R}; L^2)$ such that 
\begin{equation}
I_n=\Big\|\frac{|\xi|^{\frac{1}{2}}\la\xi\ra^\frac{d-1}{2}\tilde{V}_n(-|\xi_1|^2+|\xi-\xi_1|^2,\xi)}{\la\xi_1\ra^{(d-1)/2}\la\xi-\xi_1\ra^{(d-1)/2}}\Big\|_{L_{\xi,\xi_1}^2}^2\to\infty.
\end{equation}
We define $V_n(t,x)$ by 
\begin{equation}
\tilde{V}_n(\tau,\xi)=\chi_1(\tau-|\xi|^2)\chi_2(\xi-ne_1),
\end{equation}
where $\chi_1(\tau)=\mathbf{1}_{[-1,1]}(\tau)$ and $\chi_2(\xi)=\mathbf{1}_{|\xi|\leq 1}$ are characteristic functions, and $e_1=(1,0,\cdots, 0)$ be a unit vector in $\mathbb{R}^d$. Then,
\begin{equation}
I_n=\cdots=\int_{\mathbb{R}^{d+1}}\Big\{\int_{\mathbb{R}^{d-1}}\frac{\la\xi\ra^{d-1}dq'}{2\la (q_1^*, q')\ra^{d-1}\la(q_1^*-|\xi|,q')\ra^{d-1}} \Big\} |\tilde{V}_n(\tau,\xi)|^2 d\xi d\tau,
\end{equation}
where $q_1^*=-\frac{\tau-|\xi|^2}{2|\xi|}$. Let $n$ be sufficiently large. In the above integral, $n-1\leq|\xi|\leq n+1$, $|q_1^*|=|\frac{\tau-|\xi|^2}{2|\xi|}|\leq \frac{1}{2(n-1)}\leq\frac{1}{2}$ and $n-\frac{3}{2}\leq|q_1^*-|\xi||=|\frac{\tau-|\xi|^2}{2|\xi|}+|\xi||\leq n+\frac{3}{2}$. Thus, we have
\begin{equation}
\begin{aligned}
\int_{\mathbb{R}^{d-1}}\frac{\la\xi\ra^{d-1}dq'}{2\la (q_1^*, q')\ra^{d-1}\la(q_1^*-|\xi|,q')\ra^{d-1}}&\sim \int_{\mathbb{R}^{d-1}}\frac{n^{d-1}dq'}{\la q'\ra^{d-1}(n^2+|q'|^2)^{(d-1)/2}}\\
&\gtrsim \int_{|q'|\leq \frac{n}{2}}\frac{n^{d-1}dq'}{\la q'\ra^{d-1}n^{d-1}}\sim \ln n.
\end{aligned}
\end{equation}
Therefore, we conclude that 
\begin{equation}
I_n\gtrsim\int_{\mathbb{R}^{d+1}} (\ln n)|\tilde{V}_n(\tau,\xi)|^2 d\xi d\tau\sim\ln n\to\infty,
\end{equation}
since $\|V_n\|_{L_t^2(\mathbb{R};L^2)}\sim 1$.
\end{proof}

\section{Conservation of relative energy for regular solutions}
\label{sec-H2encons}

In this section, we establish local well-posedness of the equation \eqref{NLS} in $\mathfrak{H}^2$ (Theorem \ref{trace LWP}), where
$\mathfrak{H}^2$ is the Banach space of trace-class self-adjoint operators equipped with the norm
\begin{equation}\label{trace H2}
\|\gamma\|_{\mathfrak{H}^2}:=\|\la\nabla\ra^2\gamma\la\nabla\ra^2\|_{\mathfrak{S}^1}.
\end{equation}
Subsequently, we show that $\mathfrak{H}^2$-regular solutions obey the  conservation law for the relative energy (Proposition \ref{conservation law for regular solutions}).
The purpose is (as we recall from the outline of proof in Section \ref{sec-statement}) that we aim to approximate the
solutions obtained in Theorem \ref{LWP0} with $\mathfrak{H}^2$-regular solutions, to
prove the conservation of the relative energy of the former.
This will be carried out in Section \ref{Conservation of the relative energy at 0T}.

Before we give the statement of our main result regarding local well-posedness  in $\mathfrak{H}^2$, we 
prove the following Lemma. 

\begin{lemma}\label{trace LWP lemma} $(i)$ For $Q\in\mathfrak{H}^2$,
\begin{equation}
\|[\rho_Q, \Pi^-]\|_{\mathfrak{S}^2}\lesssim \|Q\|_{\mathfrak{S}^2}.
\end{equation}
$(ii)$ For $Q_1, Q_2\in\mathfrak{H}^2$,
\begin{equation}
\|[\rho_{Q_1},Q_2]\|_{\mathfrak{H}^2}\lesssim\|\rho_{Q_1}\|_{H^2}\|Q_2\|_{\mathfrak{H}^2}\lesssim\|Q_1\|_{\mathfrak{H}^2}\|Q_2\|_{\mathfrak{H}^2}.
\end{equation}
\end{lemma}

\begin{proof}
$(i)$ By the trivial inequality $|[A,B]|\leq 2|AB|$ and the Leibnitz rule, we write
\begin{equation}
\begin{aligned}
\|[\rho_Q, \Pi^-]\|_{\mathfrak{H}^2}&\leq 2\|(1-\Delta)\rho_{Q}\Pi^-(1-\Delta)\|_{\mathfrak{S}^1}\\
&\leq 2\Big\{\|(-\Delta\rho_{Q})(\Pi^-(1-\Delta))\|_{\mathfrak{S}^1}+2\sum_{k=1}^d\|(\partial_{x_k}\rho_{Q})(\partial_{x_k}\Pi^-(1-\Delta))\|_{\mathfrak{S}^1}\\
&\quad\quad +\|\rho_{Q}((1-\Delta)\Pi^-(1-\Delta))\|_{\mathfrak{S}^1}\Big\}.
\end{aligned}
\end{equation}
Then, using the Birman-Solomjak inequality (see Theorem 4.5 in \cite{Simon})
\begin{equation}
\|f(x)g(-i\nabla)\|_{\mathfrak{S}^1}\lesssim \|f\|_{\ell^1 L^2}\|g\|_{\ell^1L^2},
\end{equation}
where $\|f\|_{\ell^1 L^2}=\sum_{z\in\mathbb{Z}^d}\|f\|_{L^2(C_z)}$ and $\{C_z\}_{z\in\mathbb{Z}^d}$ is the collection of cubes $C_z:=z+[0,1)^d$, and using the eigenfunction expansion 
$$
	\la\nabla\ra^2Q\la\nabla\ra^2=\sum_{j=1}^\infty\lambda_j |\phi_j\ra\la\phi|,
$$
where $\{\phi_j\}_{j=1}^\infty$ is an orthonormal set in $L^2$, we obtain
\begin{equation}
\begin{aligned}
\|[\rho_Q, \Pi^-]\|_{\mathfrak{H}^2}&\lesssim \|\Delta\rho_{Q}\|_{\ell^1 L^2}+\|\nabla\rho_{Q}\|_{\ell^1 L^2}+\|\rho_{Q}\|_{\ell^1 L^2} \\
& \lesssim \|(1-\Delta)\rho_{Q}\|_{\ell^1 L^2}\\
&=\Big\|(1-\Delta)\sum_{j=1}^\infty\lambda_j|\la\nabla\ra^{-2}\phi_j|^2\Big\|_{\ell^1L^2}\\
& \leq \sum_{j=1}^\infty |\lambda_j| \big\|(1-\Delta)\big(|\la\nabla\ra^{-2}\phi_j|^2\big)\big\|_{\ell^1L^2}.
\end{aligned}
\end{equation}
By the Leibnitz rule, the H\"older inequality and Sobolev inequality, 
\begin{equation}
\begin{aligned}
&\big\|(1-\Delta)\big(|\la\nabla\ra^{-2}\phi_j|^2\big)\big\|_{\ell^1L^2}\\
&\leq\big\|\la\nabla\ra^{-2}\phi_j\big\|_{\ell^2L^4}^2+2\big\|\Delta\la\nabla\ra^{-2}\phi_j\big\|_{L^2}\big\|\la\nabla\ra^{-2}\phi_j\big\|_{\ell^2L^\infty}+2\big\|\nabla\la\nabla\ra^{-2}\phi_j\big\|_{\ell^2L^4}^2\\
&\lesssim\|\phi_j\|_{L^2}^2=1.
\end{aligned}
\end{equation}
Thus, we conclude that $\|[\rho_Q, \Pi^-]\|_{\mathfrak{H}^2}\lesssim \sum_{j=1}^\infty |\lambda_j|=\|Q\|_{\mathfrak{H}^2}$.

$(ii)$ Similarly, by the Leibnitz rule, we write
\begin{equation}
\begin{aligned}
\big\|[\rho_{Q_1}Q_2\big\|_{\mathfrak{H}^2}&\leq \|(\Delta\rho_{Q_1})Q_2\la\nabla\ra^2\|_{\mathfrak{S}^1}+2\|(\nabla\rho_{Q_1})\cdot\nabla Q_2\la\nabla\ra^2\|_{\mathfrak{S}^1}+\|\rho_{Q_1}\la\nabla\ra^2Q_2\la\nabla\ra^2\|_{\mathfrak{S}^1}\\
&\leq \Big\{\|(\Delta\rho_{Q_1})\la\nabla\ra^{-2}\|_{\textup{Op}}+2\|(\nabla\rho_{Q_1})\la\nabla\ra^{-1}\|_{\textup{Op}}+\|\rho_{Q_1}\|_{\textup{Op}}\Big\}\|Q_2\|_{\mathfrak{H}^2}\\
&\leq \Big\{\|\Delta\rho_{Q_1}\|_{L^2}\|\la\nabla\ra^{-2}\|_{L^2\to L^\infty}+2\|\nabla\rho_{Q_1}\|_{L^{d^+}}\|\la\nabla\ra^{-1}\|_{L^{(\frac{2d}{d-2})^-}\to L^2}\\
&\quad\quad\quad+\|\rho_{Q_1}\|_{L^\infty}\Big\}\|Q_2\|_{\mathfrak{H}^2}\\
&\lesssim \|(1-\Delta)\rho_{Q_1}\|_{L^2}\|Q_2\|_{\mathfrak{H}^2}\quad\textup{(by the Sobolev inequality)}.
\end{aligned}
\end{equation}
Then, using the canonical form as in $(i)$, we complete the proof.
\end{proof}

The well-posedness result of this section can be formulated as follows: 

\begin{theorem}[Local well-posedness in $\mathfrak{H}^2$]\label{trace LWP}
Let $d=2,3$. For initial data $Q_0\in\mathfrak{H}^2$, there exists a unique solution $Q\in C_t(I;\mathfrak{H}^2)$ to the equation \eqref{NLS}.
\end{theorem}

\begin{proof}[Proof of Theorem \ref{trace LWP}]
We define
\begin{equation}
\Phi(Q):=e^{it\Delta}Q_0 e^{-it\Delta}-i\int_0^t e^{i(t-s)\Delta}[\rho_Q, \Pi^-+Q](s)e^{-i(t-s)\Delta}ds.
\end{equation}
By the standard contraction mapping argument, it suffices to show that $\Phi$ is contractive on the ball $B_R\subset C_t(I;\mathfrak{H}^2)$ of radius $R=2\|Q_0\|_{\mathfrak{H}^2}$, where $I\ni 0$ is a short time interval to be chosen later. 

First, by the Minkowski inequality and using the unitarity of the linear propagator, we write
\begin{equation}\label{trace LWP proof}
\begin{aligned}
\|\Phi(Q)\|_{C_t(I;\mathfrak{H}^2)}&\leq \|Q_0\|_{\mathfrak{H}^2}+\|[\rho_Q,\Pi^-]\|_{L_t^1(I;\mathfrak{H}^2)}+\|[\rho_Q, Q]\|_{L_t^1(I;\mathfrak{H}^2)},\\
\|\Phi(Q_1)-\Phi(Q_2)\|_{C_t(I;\mathfrak{H}^2)}&\leq\|[\rho_{(Q_1-Q_2)},\Pi^-]\|_{L_t^1(I;\mathfrak{H}^2)}+\|[\rho_{(Q_1-Q_2)}, Q_1]\|_{L_t^1(I;\mathfrak{H}^2)}\\
&\quad\quad+\|[\rho_{Q_2}, (Q_1-Q_2)]\|_{L_t^1(I;\mathfrak{H}^2)},
\end{aligned}
\end{equation}
where the identity 
$$
	[\rho_{Q_1}, Q_1]-[\rho_{Q_2}, Q_2]=[\rho_{Q_1-Q_2}, Q_1]+[\rho_{Q_2}, Q_1-Q_2]
$$ is used for the difference. Then, applying Lemma \ref{trace LWP lemma}, we get
\begin{equation}\label{choice of c in trace LWP}
\begin{aligned}
\|\Phi(Q)\|_{C_t(I;\mathfrak{H}^2)}&\leq \|Q_0\|_{\mathfrak{H}^2}+c|I|\|Q\|_{C_t(I;\mathfrak{H}^2)}+c|I|\|Q\|_{C_t(I;\mathfrak{H}^2)}^2,\\
\|\Phi(Q_1)-\Phi(Q_2)\|_{C_t(I;\mathfrak{H}^2)}&\leq c|I|\Big\{1+|\|Q_1\|_{C_t(I;\mathfrak{H}^2)}+\|Q_2\|_{C_t(I;\mathfrak{H}^2)}\Big\}\|Q_1-Q_2\|_{C_t(I;\mathfrak{H}^2)}.
\end{aligned}
\end{equation}
Now, choosing an interval $I$ with length $I=\min\{\frac{1}{4c}, \frac{1}{8c\|Q_0\|_{\mathfrak{H}^2}}\}$, we conclude that 
\begin{equation}
\begin{aligned}
\|\Phi(Q)\|_{C_t(I;\mathfrak{H}^2)}&\leq 2\|Q_0\|_{\mathfrak{H}^2},\\
\|\Phi(Q_1)-\Phi(Q_2)\|_{C_t(I;\mathfrak{H}^2)}&\leq \frac{3}{4}\|Q_1-Q_2\|_{C_t(I;\mathfrak{H}^2)}
\end{aligned}
\end{equation}
for $Q,Q_1,Q_2\in B_R$.
\end{proof}

Next, we prove the conservation law for the regular solutions.

\begin{proposition}[Conservation of the relative energy for regular solutions]\label{conservation law for regular solutions}
Let $d=2,3$. If $Q(t)\in C_t(I;\mathfrak{H}^2)$ solves the equation \eqref{NLS} with initial data $Q_0$, then
\begin{equation}
\mathcal{E}(Q(t))=\mathcal{E}(Q_0),\quad\forall t\in I.
\end{equation}
\end{proposition}

\begin{proof}
Differentiating the relative energy, and then inserting the equation \eqref{NLS}, we get
\begin{equation}
\begin{aligned}
\frac{d}{dt}\mathcal{E}(Q(t))&=\textup{Tr}\Big\{(-\Delta-1)\partial_tQ+\rho_{Q}\partial_tQ\Big\}(t)\\
&=(-i)\textup{Tr}\Big\{(-\Delta-1)[-\Delta+\rho_Q,Q]+\rho_Q[-\Delta+\rho_Q,Q]\Big\}(t).
\end{aligned}
\end{equation}
Indeed, the above calculation make sense, since $Q(t)\in \mathfrak{H}^2$. By cyclicity of the trace, we have $\textup{Tr}(-\Delta-1)[-\Delta, Q]=0$, $\textup{Tr}[\rho_Q,Q]=0$ and $\textup{Tr}\rho_Q[\rho_Q,Q]=0$. Thus, we get
\begin{equation}
\frac{d}{dt}\mathcal{E}(Q(t))=(-i)\Big\{\textup{Tr}(-\Delta)[\rho_Q,Q]+\textup{Tr}\rho_Q[-\Delta,Q]\Big\}(t).
\end{equation}
Then, by cyclicity again,
\begin{equation}
\begin{aligned}
&\textup{Tr}(-\Delta)[\rho_Q,Q]+\textup{Tr}\rho_Q[-\Delta,Q]\\
&=\textup{Tr}(-\Delta)\rho_Q Q-\textup{Tr}(-\Delta)Q\rho_Q+\textup{Tr}\rho_Q(-\Delta)Q-\textup{Tr}\rho_QQ(-\Delta)\\
&=\textup{Tr}(-\Delta)\rho_Q Q-\textup{Tr}(-\Delta)Q\rho_Q+\textup{Tr}(-\Delta)Q\rho_Q-\textup{Tr}(-\Delta)\rho_QQ=0.
\end{aligned}
\end{equation}
Therefore, we conclude that $\frac{d}{dt}\mathcal{E}(Q(t))=0$.
\end{proof}

\section{Local well-posedness in a space larger than the energy space}
\label{sec-XYspaceLWP}

In this section, we prove the main local well-posedness results contained in Theorem \ref{LWP0}. 

For the convenience of the reader, we review the norms introduced in Section \ref{sec-statement}. 
Let $\alpha\geq 1$. The initial data space $\mathfrak{X}^\alpha$ was defined as the Banach space of self-adjoint operators equipped the norm
\begin{equation}
\|Q\|_{\mathfrak{X}^\alpha}:=\|Q\|_{\textup{Op}}+\|\Pi_2^+Q\|_{\mathcal{H}^\alpha}.
\end{equation}
Given $I\subset\mathbb{R}$, the solution space $\mathfrak{Y}^\alpha(I)$ was defined as the Banach space of time-dependent self-adjoint operators with the norm
\begin{equation}\label{Y alpha norm}
\|Q\|_{\mathfrak{Y}^\alpha(I)}:=\|Q\|_{C_t(I;\textup{Op})}+\|\Pi_2^+Q\|_{\mathcal{S}^\alpha(I)}+\|\rho_Q\|_{L_t^2(I;H^{\alpha+\frac{1}{2}-\eta})\bigcap L_t^\infty(I; L^2)},
\end{equation}
where
\begin{equation}\label{eta}
\left\{\begin{aligned}
&\eta>0&&\text{ when }d=3\textup{ and }\alpha=1,\\
&\eta=0&&\text{ when }d=2\textup{ and }\alpha\geq 1\textup{, or }d=3\textup{ and }\alpha>1.
\end{aligned}
\right.
\end{equation}
The reason we need $\eta$ in the norm is that $\alpha_1$ in Theorem \ref{Strichartz estimates for density functions} is strictly less than 1 when $\alpha=1$ and $d=3$. This is to offset the logarithmic growth of \eqref{integral 1} and \eqref{integral 2} in $|\xi|$ which appears in this case.

Note that when $\alpha=1$, Theorem \ref{LWP0} provides local well-posedness in a space larger than the relative kinetic energy space. Indeed, the initial data space $\mathfrak{X}^1$ contains the relative kinetic energy space $\mathcal{K}$, since by the estimates in \eqref{energy space proof 2},
\begin{equation}
\|\Pi_2^+ Q\|_{\mathcal{H}^1}\leq\|\Pi_2^+ Q\Pi_2^+\|_{\mathfrak{H}^1}+\|\Pi_2^+ Q\Pi_2^-\|_{\mathcal{H}^1}<\infty.
\end{equation}

%We prove the following local well-posedness theorem.

%\begin{theorem}[Local well-posedness]\label{LWP}
%Let $d=2,3$ and $\alpha\geq 1$. Given initial data $Q_0\in\mathfrak{X}^\alpha\cap\mathcal{K}$, there exists a unique solution $Q\in\mathfrak{Y}^\alpha(I)$ to the equation \eqref{NLS}.
%\end{theorem}

The main tools to prove Theorem \ref{LWP0} are the Strichartz estimates in Section \ref{sec-strichartz}. Note that the Strichartz estimates in Theorem \ref{Strichartz estimates for density functions} cannot control the $L^2$-norm of (the low frequency part of) the density function $\rho_{e^{it\Delta}Q e^{-it\Delta}}$, while we need it for global well-posedness, since the potential energy of the relative energy is $\|\rho_Q\|_{L^2}^2$. Hence, we modify Strichartz estimates using the Lieb-Thirring inequality in Frank, Lewin, Lieb and Seiringer \cite{FLLS1},
\begin{equation}\label{LT inequality}
\textup{Tr}_0(-\Delta-1)Q\geq K_\textup{LT}\int_{\mathbb{R}^d}\Big\{\big(\rho_{\Pi^-}+\rho_Q\big)^{1+\frac{2}{d}}-\big(\rho_{\Pi^-}\big)^{1+\frac{2}{d}}-\tfrac{2+d}{d}\big(\rho_{\Pi^-}\big)^{\frac{2}{d}}\rho_Q\Big\}dx
\end{equation}
and the generalized Sobolev inequality in \cite{CP1} 
\begin{equation}\label{Sobolev inequality}
\|\rho_Q\|_{L^2}\lesssim \|Q\|_{\mathcal{H}^1}.
\end{equation}

\begin{proposition}[Local-in-time Strichartz estimates for density functions]\label{local-in-time Strichartz}
Let $d=2,3$, $\alpha\geq 1$ and $\eta$ be as in \eqref{eta}, and let $I\ni 0$ be an interval with $|I|\leq1$. Then, we have
\begin{equation}\label{homogeneous Strichartz estimates for density functions at 0T}
\|\rho_{e^{it\Delta}Q_0e^{-it\Delta}}\|_{L_t^2(I;H^{\alpha+\frac{1}{2}-\eta})\bigcap L_t^\infty(I; L^2)}\lesssim\|Q_0\|_{\mathfrak{X}^\alpha}+\Big\{\textup{Tr}_0(-\Delta-1)Q_0\Big\}^{\frac{1}{2}}
\end{equation}
and
\begin{equation}\label{inhomogeneous Strichartz estimates for density functions at 0T}
\Big\|\rho\Big[ \int_0^t e^{i(t-s)\Delta}R(s) e^{-i(t-s)\Delta}ds \Big]\Big\|_{L_t^2(I;H^{\alpha+\frac{1}{2}-\eta})\bigcap L_t^\infty(I; L^2)}\lesssim\|R(t)\|_{L_t^1(I;\mathcal{H}^\alpha)}.
\end{equation}
\end{proposition} 

\begin{proof}
We first prove \eqref{homogeneous Strichartz estimates for density functions at 0T}. 
We shall estimate the l.h.s. of \eqref{homogeneous Strichartz estimates for density functions at 0T} via splitting the operator 
$Q_0$ into:
high-low $\Pi_2^+Q_0\Pi_2^-$, 
low-high $\Pi_2^-Q_0\Pi_2^+$, 
high-high $\Pi_2^+Q_0\Pi_2^+$ and 
low-low  $\Pi_2^-Q_0\Pi_2^-$
frequency parts. 

In order to consider high-low, low-high and high-high frequency parts of the operator $Q_0$ we first observe that: 
\begin{equation} \label{revS-split} 
\Pi_2^+Q_0\Pi_2^-+(\Pi_2^-Q_0\Pi_2^++\Pi_2^+Q_0\Pi_2^+)=\Pi_2^+Q_0\Pi_2^-+Q_0\Pi_2^+.
\end{equation} 

Now we present the estimate that will be useful when we bound the terms coming from the r.h.s. of \eqref{revS-split}. 
More precisely, we have: 
\begin{equation}\label{local-in-time Strichartz at 0T proof}
\begin{aligned}
&\|\rho_{e^{it\Delta}Q_0e^{-it\Delta}}\|_{L_t^2(I;H^{\alpha+\frac{1}{2}-\eta})\bigcap L_t^\infty(I; L^2)}\\
&\lesssim\||\nabla|^{\alpha+\frac{1}{2}-\eta}\rho_{e^{it\Delta}Q_0e^{-it\Delta}}\|_{L_t^2(I;L^2)}+\|\rho_{e^{it\Delta}Q_0e^{-it\Delta}}\|_{L_t^\infty(I; L^2)}\\
&\lesssim \|Q_0\|_{\mathcal{H}^\alpha},
\end{aligned}
\end{equation}
where in the last inequality, we used the Strichartz estimate \eqref{homogeneous Strichartz estimates for density functions} for the first term, 
and Sobolev inequality \eqref{Sobolev inequality} for the second term. 

Then we estimate terms coming from the r.h.s. of \eqref{revS-split} by
using \eqref{local-in-time Strichartz at 0T proof} as follows: 
\begin{equation}
\begin{aligned}
&\|\rho_{e^{it\Delta}(\Pi_2^+Q_0\Pi_2^-+Q_0\Pi_2^+)e^{-it\Delta}}\|_{L_t^2(I;H^{\alpha+\frac{1}{2}-\eta})\bigcap L_t^\infty(I; L^2)}\\
&\lesssim \|\Pi_2^+Q_0\Pi_2^-+Q_0\Pi_2^+\|_{\mathcal{H}^\alpha}
\\
&\leq\|\Pi_2^+Q_0\Pi_2^-\|_{\mathcal{H}^\alpha}+\|Q_0\Pi_2^+\|_{\mathcal{H}^\alpha}\\
&\leq 2\|\Pi_2^+Q_0\|_{\mathcal{H}^\alpha}\quad\textup{(by symmetry)}.
\end{aligned}
\end{equation}

It remains to consider the low-low frequency part $\Pi_2^-Q_0\Pi_2^-$. We claim that in general, the Sobolev norm of the density function of the low-low frequency part is bounded by its $L^2$ norm. In other words, 
\begin{equation}\label{eq-rhoPi2-claim-1}
\|\rho_{\Pi_2^-Q\Pi_2^-}\|_{H^\alpha}\lesssim \|\rho_{\Pi_2^-Q\Pi_2^-}\|_{L^2}.
\end{equation}
Indeed, by the Plancherel theorem, 
\begin{equation}
\begin{aligned}
\|\rho_{\Pi_2^-Q\Pi_2^-}\|_{H^\alpha}&=\Big\|\la\xi\ra^\alpha\int_{\mathbb{R}^d} (\Pi_2^-Q\Pi_2^-)^\wedge(\xi-\xi',\xi')d\xi'\Big\|_{L_\xi^2}\\
&\lesssim \Big\|\int_{\mathbb{R}^d} (\Pi_2^-Q\Pi_2^-)^\wedge(\xi-\xi',\xi')d\xi'\Big\|_{L_\xi^2}=\|\rho_{\Pi_2^-Q\Pi_2^-}\|_{L^2},
\end{aligned}
\end{equation}
where in the first inequality, we used $|\xi-\xi'|,|\xi'|\leq2$ $(\Rightarrow |\xi|\leq 4)$.  Therefore, using \eqref{eq-rhoPi2-claim-1} (with the fact that $e^{\pm it\Delta}$ commutes with $\Pi_2^-$), 
the estimate \eqref{energy space proof 1} (which follows from the Lieb-Thirring inequality) and unitarity of the linear propagator, we prove
\begin{equation}
\begin{aligned}
&\|\rho_{e^{it\Delta}\Pi_2^-Q_0\Pi_2^-e^{-it\Delta}}\|_{L_t^2(I;H^{\alpha+\frac{1}{2}-\eta})\bigcap L_t^\infty(I; L^2)}\\
&\lesssim \|\rho_{e^{it\Delta}\Pi_2^-Q_0\Pi_2^-e^{-it\Delta}}\|_{L_t^\infty(I; L^2)}\\
&\lesssim\sup_{t\in I}\Big\{\textup{Tr}_0(-\Delta-1)e^{it\Delta}\Pi_2^-Q_0\Pi_2^-e^{-it\Delta}\Big\}^{\frac{1}{2}}\\
&=\Big\{\textup{Tr}_0(-\Delta-1)\Pi_2^-Q_0\Pi_2^-\Big\}^{\frac{1}{2}}\leq\Big\{\textup{Tr}_0(-\Delta-1)Q_0\Big\}^{\frac{1}{2}}.
\end{aligned}
\end{equation}
This establishes the proof of  \eqref{homogeneous Strichartz estimates for density functions at 0T}. 

Now we prove \eqref{inhomogeneous Strichartz estimates for density functions at 0T}.
By the Minkowski inequality and \eqref{local-in-time Strichartz at 0T proof}, we prove that the left hand side of \eqref{inhomogeneous Strichartz estimates for density functions at 0T} is bounded by
\begin{equation}
\begin{aligned}
&\int_I \|\rho_{e^{it\Delta}(e^{-is\Delta}R(s)e^{is\Delta}) e^{-it\Delta}}\|_{L_t^2(I;H^{\alpha+\frac{1}{2}-\eta})\bigcap L_t^\infty(I; L^2)}ds \\
&\lesssim\int_I \|e^{-is\Delta}R(s) e^{is\Delta}\|_{\mathcal{H}^\alpha}ds=\|R(t)\|_{L_t^1(I;\mathcal{H}^\alpha)}.
\end{aligned}
\end{equation}
\end{proof}

The following lemma is analogous to Lemma \ref{trace LWP lemma}.
\begin{lemma}\label{LWP lemma}
Let $\alpha\geq1$. For $Q,Q_1,Q_2 \in \mathfrak{Y}^\alpha(I)$, we have the following:
\\
$(i)$ 
\begin{equation} \label{5com-i}
\|[\rho_{Q_1}, Q_2]\|_{L_t^1(I;\mathcal{H}^\alpha)}\lesssim |I|^{1/8}\|\rho_{Q_1}\|_{L_t^2(I; H^\alpha)}\|Q_2\|_{\mathcal{S}^\alpha(I)}.
\end{equation}
$(ii)$
\begin{equation} \label{5com-ii}
\|[\rho_Q,\Pi^-]\|_{L_t^1(I;\mathcal{H}^\alpha)}\lesssim |I|^{\frac{1}{2}}\|\rho_Q\|_{L_t^2(I;H^\alpha)}.
\end{equation}
\end{lemma}

\begin{proof}
$(i)$ First, we consider the term with the low-low frequency part of $Q_2$, that is, $\Pi_2^-Q_2\Pi_2^-$. 
To do that we recall the definition of  the Hilbert-Schmidt type Sobolev space 
$\mathcal{H}^\alpha$ (see \eqref{def-HSSob})
and use H\"{o}lder inequality as follows:  
\begin{equation}
\begin{aligned}
&\|[\rho_{Q_1}, \Pi_2^-Q_2\Pi_2^-]\|_{L_t^1(I;\mathcal{H}^\alpha)}\\
&\leq2\|\rho_{Q_1} \Pi_2^-Q_2\Pi_2^-\|_{L_t^1(I;\mathcal{H}^\alpha)} \\
& =2\|\la\nabla\ra^\alpha\rho_{Q_1} \Pi_2^-Q_2\Pi_2^-\la\nabla\ra^\alpha\|_{L_t^1(I;\mathfrak{S}^2)}\\
&\leq2\|\la\nabla\ra^\alpha\rho_{Q_1} \Pi_2^-\|_{L_t^1(I;\mathfrak{S}^2)}\|Q_2\|_{C_t(I;\textup{Op})}\|\Pi_2^-\la\nabla\ra^\alpha\|_{\textup{Op}}\\
&=2\|\la\nabla_x\ra^\alpha(\rho_{Q_1}(x)\Pi_2^-(x-x'))\|_{L_t^1(I;L_x^2L_{x'}^2)}\|Q_2\|_{C_t(I;\textup{Op})}\|\Pi_2^-\la\nabla\ra^\alpha\|_{\textup{Op}}.
\end{aligned}
\end{equation}
Hence, by the fractional Leibnitz rule (Theorem A.8 in \cite{KPV}), we get
\begin{equation}
\begin{aligned}
\|[\rho_{Q_1}, \Pi_2^-Q_2\Pi_2^-]\|_{L_t^1(I;\mathcal{H}^\alpha)}&\lesssim |I|^{\frac{1}{2}}\|\rho_{Q_1}\|_{L_t^2(I; H^\alpha)}\|\Pi_2^-(x)\|_{H^\alpha}\|Q_2\|_{C_t(I;\textup{Op})}\\
&\lesssim |I|^{\frac{1}{2}}\|\rho_{Q_1}\|_{L_t^2(I; H^\alpha)}\|Q_2\|_{C_t(I;\textup{Op})},
\end{aligned}
\end{equation}
where $\Pi_2^-(x)$ is the kernel of the operator $\Pi_2^-$.

Next, we consider the remainder with $\tilde{Q}_2:=Q_2-\Pi_2^-Q_2\Pi_2^-$. By the fractional Leibnitz rule (Theorem A.8 in \cite{KPV}) and the Sobolev inequality,
\begin{equation} \label{5com-1}
\begin{aligned}
\|[\rho_{Q_1}, \tilde{Q}]\|_{\mathcal{H}^\alpha}&\leq 2\|\rho_{Q_1}\tilde{Q}\|_{\mathcal{H}^\alpha}\\
& =2\big\|\la\nabla_x\ra^\alpha\la\nabla_{x'}\ra^\alpha \big(\rho_{Q_1}(x) \tilde{Q}_2(x,x')\big)\big\|_{L_x^2L_{x'}^2}\\
&\lesssim \|\rho_{Q_1}\|_{H^\alpha}\|\la\nabla_{x'}\ra^\alpha\tilde{Q}_2\|_{L_x^\infty L_{x'}^2}+\|\rho_{Q_1}\|_{L_x^{\frac{2d}{d-2+\epsilon}}}\|\la\nabla_x\ra^\alpha\la\nabla_{x'}\ra^\alpha\tilde{Q}_2\|_{L_x^{\frac{2d}{2-\epsilon}} L_{x'}^2}\\
&\lesssim \|\rho_{Q_1}\|_{H^\alpha}\|\la\nabla_x\ra^\alpha\la\nabla_{x'}\ra^\alpha\tilde{Q}_2\|_{L_x^{\frac{2d}{2-\epsilon}} L_{x'}^2},
\end{aligned}
\end{equation}
where $\epsilon>0$ is sufficiently small. Then, further splitting $\tilde{Q}_2$ 
$$
	\tilde{Q}_2=
	(\Pi_2^+Q_2\Pi_2^+ + \Pi_2^+Q_2\Pi_2^-) + \Pi_2^-Q_2\Pi_2^+
	= \Pi_2^+Q_2+\Pi_2^-Q_2\Pi_2^+
$$ and dropping $\Pi_2^-$ in the second term, 
\begin{equation} \label{5com-2}
\begin{aligned}
&\|\la\nabla_x\ra^\alpha\la\nabla_{x'}\ra^\alpha\tilde{Q}_2\|_{L_x^{\frac{2d}{2-\epsilon}} L_{x'}^2}\\
&\leq\|\rho_{Q_1}\|_{H^\alpha}\Big\{\|\la\nabla_x\ra^\alpha\la\nabla_{x'}\ra^\alpha\Pi_2^+Q_2\|_{L_x^{\frac{2d}{2-\epsilon}} L_{x'}^2}+\|\la\nabla_x\ra^\alpha\la\nabla_{x'}\ra^\alpha \Pi_2^-Q_2\Pi_2^+\|_{L_x^{\frac{2d}{2-\epsilon}} L_{x'}^2}\Big\}\\
&\leq\|\rho_{Q_1}\|_{H^\alpha}\Big\{\|\la\nabla_x\ra^\alpha\la\nabla_{x'}\ra^\alpha\Pi_2^+Q_2\|_{L_x^{\frac{2d}{2-\epsilon}} L_{x'}^2}+\|\la\nabla_x\ra^\alpha\la\nabla_{x'}\ra^\alpha Q_2\Pi_2^+\|_{L_x^{\frac{2d}{2-\epsilon}} L_{x'}^2}\Big\}\\
&=2\|\rho_{Q_1}\|_{H^\alpha}\|\la\nabla_x\ra^\alpha\la\nabla_{x'}\ra^\alpha\Pi_2^+Q_2\|_{L_x^{\frac{2d}{2-\epsilon}} L_{x'}^2}\quad\textup{(by symmetry)}.
\end{aligned}
\end{equation}
Combining \eqref{5com-1} and \eqref{5com-2}, we conclude after performing H\"{o}lder with respect to $t$  that 
\begin{equation}
\begin{aligned}
\|[\rho_{Q_1}, Q_2]\|_{L_t^1(I;\mathcal{H}^\alpha)}&\lesssim |I|^{\frac{4-d-\epsilon}{4}}\|\rho_{Q_1}\|_{L_t^2(I;H^\alpha)}\|\la\nabla_x\ra^\alpha\la\nabla_{x'}\ra^\alpha\Pi_2^+Q_2\|_{L_t^{\frac{4}{d-2+\epsilon}}(I;L_x^{\frac{2d}{2-\epsilon}} L_{x'}^2)}\\
&\leq|I|^{1/8}\|\rho_{Q_1}\|_{L_t^2(I;H^\alpha)}\|\Pi_2^+Q_2\|_{\mathcal{S}^\alpha(I)},
\end{aligned}
\end{equation}
where the last inequality follows thanks to  $(\frac{4}{d-2+\epsilon},\frac{2d}{2-\epsilon})$ being an admissible pair. 
Hence \eqref{5com-i} is proved. 

$(ii)$ Similarly, we prove that 
\begin{equation}\label{LWP proof claim 2}
\begin{aligned}
\|[\rho_Q,\Pi^-]\|_{L_t^1(I;\mathcal{H}^\alpha)}&\leq 2\|\rho_Q \Pi^-\|_{L_t^1(I;\mathcal{H}^\alpha)}\\
&=2\big\|\la\nabla_x\ra^\alpha\big(\rho_Q(x) (\Pi^-\la\nabla\ra^\alpha)(x-x')\big)\big\|_{L_t^1(I;L_x^2L_{x'}^2)}\\
&\lesssim |I|^{\frac{1}{2}}\|\rho_Q\|_{L_t^2(I;H^\alpha)}\|\Pi^-(x)\|_{H^{2\alpha}}\\
&\lesssim |I|^{\frac{1}{2}}\|\rho_Q\|_{L_t^2(I;H^\alpha)}.
\end{aligned}
\end{equation}
\end{proof}

\begin{proof}[Proof of Theorem \ref{LWP0}]
Let $\alpha\geq1$. Let $I\ni 0$ be a sufficiently short interval $(|I|\leq 1)$ whose length will be chosen later. We define the nonlinear mapping $\Phi$ by
\begin{equation}
\Phi(Q)=e^{it\Delta}Q_0e^{-it\Delta}-i\int_0^t e^{i(t-s)\Delta}[\rho_Q, \Pi^-+Q](s)e^{-i(t-s)\Delta}ds.
\end{equation}
Then, by a standard contraction mapping argument, it suffices to show that $\Phi$ is contractive on a ball in $\mathfrak{Y}^\alpha(I)$.

Applying trivial estimates for $\Phi(Q)$ in the operator norm and Strichartz estimates (Theorem \ref{Strichartz estimates for operator kernels} for $\Pi_2^+\Phi(Q)$ and Proposition \ref{local-in-time Strichartz} for $\rho_{\Phi(Q)}$), we get
\begin{equation}\label{LWP proof}
\begin{aligned}
\|\Phi(Q)\|_{\mathfrak{Y}^\alpha(I)}&=\|\Phi(Q)\|_{C_t(I;\textup{Op})}+\|\Pi_2^+\Phi(Q)\|_{\mathcal{S}^\alpha(I)}+\|\rho_{\Phi(Q)}\|_{L_t^\infty(I; L^2)\bigcap L_{t}^2 (I;H_x^{\alpha+\frac{1}{2}-\eta})}\\
&\lesssim \Big\{\|Q_0\|_{\textup{Op}}+\|[\rho_Q, \Pi^-+Q]\|_{L_t^1(I;\textup{Op})}\Big\}
\\
&\quad+ \Big\{\|\Pi_{2}^+Q_0\|_{\mathcal{H}^\alpha}+\|\Pi_{2}^+[\rho_Q, \Pi^-+Q]\|_{L_t^1(I;\mathcal{H}^\alpha)}\Big\}\\
&\quad+\Big\{A+\|[\rho_Q, \Pi^-+Q]\|_{L_t^1(I;\mathcal{H}^\alpha)}\Big\}\\
&\lesssim A+\|[\rho_Q, \Pi^-]\|_{L_t^1(I;\mathcal{H}^\alpha)}+\|[\rho_Q, Q]\|_{L_t^1(I;\mathcal{H}^\alpha)},
\end{aligned}
\end{equation}
where
\begin{equation}
A=\|Q_0\|_{\mathfrak{X}^\alpha}+\{\textup{Tr}_0(-\Delta-1)Q_0\}^{\frac{1}{2}}.
\end{equation}
Then, by Lemma \ref{LWP lemma}, we obtain
\begin{equation}\label{LWP proof 2}
\|\Phi(Q)\|_{\mathfrak{Y}^\alpha(I)}\leq cA+c|I|^{\frac{1}{2}}\|\rho_Q\|_{L_t^2(I;H^{\alpha})}+c|I|^{1/8}\|\rho_Q\|_{L_t^2(I;H^{\alpha})} \|Q\|_{\mathcal{S}^\alpha(I)}.
\end{equation}
In the same way, we show that 
\begin{equation}\label{LWP proof 3}
\begin{aligned}
&\|\Phi(Q_1)-\Phi(Q_2)\|_{\mathfrak{Y}^\alpha(I)}\\
&=\Big\|\int_0^t e^{i(t-s)\Delta}\Big\{[\rho_{Q_1-Q_2}, \Pi^-+Q_1](s)+[\rho_{Q_2}, Q_1-Q_2](s)\Big\}e^{-i(t-s)\Delta}ds\Big\|_{\mathcal{S}^\alpha(I)}\\
&\leq c|I|^{\frac{1}{2}}\|\rho_{(Q_1-Q_2)}\|_{L_t^2(I;H^{\alpha})}+c|I|^{1/8}\|\rho_{(Q_1-Q_2)}\|_{L_t^2(I;H^{\alpha})}\|Q_1\|_{\mathcal{S}^\alpha(I)}\\
&\quad+c|I|^{1/8}\|\rho_{Q_2}\|_{L_t^2(I;H^{\alpha})}\|Q_1-Q_2\|_{\mathcal{S}^\alpha(I)}.
\end{aligned}
\end{equation}
Now, let $B_R$ be the ball of radius $R$ in $\mathfrak{Y}^\alpha(I)$, where $R=2cA$ and $|I|=\min\{\frac{1}{(4c)^2}, \frac{1}{(8c^2A)^8}\}$. The norms in the bounds of \eqref{LWP proof 2} and \eqref{LWP proof 3} are bounded by their $\mathfrak{Y}^\alpha$-norms (see \eqref{Y alpha norm}). Therefore, it follows that
\begin{equation}
\begin{aligned}
\|\Phi(Q)\|_{\mathfrak{Y}^\alpha(I)}&\leq R,\\
\|\Phi(Q_1)-\Phi(Q_2)\|_{\mathfrak{Y}^\alpha(I)}&\leq\frac{3}{4}\|Q_1-Q_2\|_{\mathfrak{Y}^\alpha(I)}
\end{aligned}
\end{equation}
if $Q,Q_1,Q_2\in B_R$. Therefore, we conclude that $\Phi$ is a contraction on $B_R$.
\end{proof}

\section{Conservation of the relative energy in the kinetic energy space}
\label{Conservation of the relative energy at 0T}

In this section we establish the conservation of the relative energy for solutions evolved initially from the relative kinetic energy space, by 
proving Theorem \ref{relative energy conservation0}.

%\begin{theorem}[Conservation of the relative energy]\label{relative energy conservation}
%Let $d=2,3$. Let $Q(t)\in\mathfrak{Y}^1(I)$ be the solution to the equation \eqref{NLS} with initial data $Q_0\in\mathcal{K}$ (constructed from Theorem \ref{LWP}). Then,
%\begin{equation}
%\mathcal{E}(Q(t))=\mathcal{E}(Q_0),\quad\forall t\in I.
%\end{equation}
%\end{theorem}

We prove Theorem \ref{relative energy conservation0} by approximating $Q(t)\in\mathfrak{Y}^1(I)$ by a sequence of regular solutions using the following approximation lemma.

\begin{lemma}[Approximation lemma]\label{approximation}\label{approximation}
Let $d=2,3$. For $Q_0\in\mathcal{K}$, there exists a sequence $\{Q^{(n)}\}_{n=1}^\infty$ of finite-rank smooth operators such that $-\Pi^-\leq Q^{(n)}\leq \Pi^+$, 
\begin{equation}
\begin{aligned}
\lim_{n\to\infty}\|Q^{(n)}-Q\|_{\textup{Op}}&=0,\\
\lim_{n\to\infty}\big\||\Delta+1|^{\frac{1}{2}}(Q^{(n)}-Q)^{\pm\pm}|\Delta+1|^{\frac{1}{2}}\big\|_{\mathfrak{S}^1}&=0,\\
\lim_{n\to\infty}\|\rho_{Q^{(n)}}-\rho_Q\|_{L^2}&=0,
\end{aligned}
\end{equation}
where $Q^{\pm\pm}=\Pi_1^\pm Q\Pi_1^\pm$.
\end{lemma}

\begin{proof}
It suffices to construct a sequence of Hilbert-Schmidt operators with the desired properties.  Let 
\begin{equation}
	P_n=\mathbf{1}_{(\frac{1}{n}\leq|\Delta+1|\leq n)}
\end{equation}
be the frequency cut-off away from the Fermi sphere $\{\xi\in\mathbb{R}^d: |\xi|=1\}$, 
and define 
\begin{equation}
Q^{(n)}:=P_nQ P_n.
\end{equation}
It is shown in \cite[Lemma 3.2]{FLLS1} and \cite[Lemma 5]{LS1} that 
\begin{equation}
\begin{aligned} 
& \|Q^{(n)}-Q\|_{\textup{Op}}\to0 \\
& \||\Delta+1|^{\frac{1}{2}}(Q^{(n)}-Q)|\Delta+1|^{\frac{1}{2}}\|_{\mathfrak{S}^1}\to0\\
& \rho_{\Pi_2^-(Q^{(n)}-Q)\Pi_2^-}\to 0 \mbox{ in } L^2.
\end{aligned}
\end{equation}
 As a consequence, by the Sobolev inequality \eqref{Sobolev inequality},
\begin{equation}
\begin{aligned}
\|\rho_{\Pi_2^+(Q^{(n)}-Q)\Pi_2^+}\|_{L^2}&\lesssim \|\Pi_2^+(Q^{(n)}-Q)\Pi_2^+\|_{\mathcal{H}^1}\\
&\lesssim \||\Delta+1|^{\frac{1}{2}}(Q^{(n)}-Q)^{++}|\Delta+1|^{\frac{1}{2}}\|_{\mathfrak{S}^1}\to 0.
\end{aligned}
\end{equation}
Note that 
\begin{equation}
\Pi_2^+(Q-Q^{(n)})\Pi_2^-=\mathbf{1}_{(|\Delta+1|\geq n)}Q\Pi_2^-+\Pi_2^+P_nQ\mathbf{1}_{(|\Delta+1|\leq\frac{1}{n})}.
\end{equation}
Hence, by the Sobolev inequality \eqref{Sobolev inequality} again,
\begin{equation}
\begin{aligned}
\|\rho_{\Pi_2^+(Q^{(n)}-Q)\Pi_2^-}\|_{L^2}&\lesssim \|\mathbf{1}_{(|\Delta+1|\geq n)}Q\Pi_2^-\|_{\mathcal{H}^1}+ \|\Pi_2^+P_nQ\mathbf{1}_{(|\Delta+1|\leq\frac{1}{n})}\|_{\mathcal{H}^1}\\
&\lesssim \||\Delta+1|^{\frac{1}{2}}\mathbf{1}_{(|\Delta+1|\geq n)}Q\|_{\mathfrak{S}^2}+ \||\Delta+1|^{\frac{1}{2}}Q\mathbf{1}_{(|\Delta+1|\leq\frac{1}{n})}\|_{\mathfrak{S}^2}\to 0,
\end{aligned}
\end{equation}
since $|\Delta+1|^{\frac{1}{2}}Q\in \mathfrak{S}^2$ if $Q\in\mathcal{K}$. Similarly, we show that $\rho_{\Pi_2^-(Q^{(n)}-Q)\Pi_2^+}\to 0$ in $L^2$.
\end{proof}

In order to approximate a solution $Q(t)\in\mathfrak{Y}^1(I)$ suitably by a sequence of regular solutions $\{Q^{(n)}(t)\}_{n=1}^\infty$, we need to justify that the interval $I^{(n)}$ of existence  for $Q^{(n)}(t)$ in $\mathfrak{H}^2$ does not shrink to $\{0\}$ as $n\to \infty$. The following lemma asserts that the interval of existence in a regular space can be extended to the interval of existence in the relative kinetic energy space.

\begin{lemma}[Existence time]
Suppose that $Q_0\in\mathfrak{H}^2\cap\mathcal{K}$. Let $Q(t)$ be the solution to the equation \eqref{NLS} with initial data $Q_0$ satisfying $\|Q(t)\|_{\mathfrak{Y}^1(I)}<\infty$, where $I$ is the interval of existence given by Theorem \ref{LWP0}. Then, $Q(t)\in C_t(I;\mathfrak{H}^2)$.
\end{lemma}

\begin{proof}
\textit{Step 1.} We claim that it suffices to show that
\begin{equation}\label{existence time: claim}
\|\rho_Q\|_{L_t^2(I;H^2)}<\infty.
\end{equation}
Indeed, by the local well-posedness theorem in the regular space $\mathfrak{H}^2$ (Theorem \ref{trace LWP}), there exists a short interval $I'\subset I$ such that $Q(t)$ exists in $C_t(I'; \mathfrak{H}^2)$. Let $t\in I'$. Then, applying Lemma \ref{trace LWP lemma} $(i)$ and the first inequality in Lemma \ref{trace LWP lemma} $(ii)$ to the Duhamel formula
\begin{equation}\label{Duhamel for existence time}
Q(t)=e^{it\Delta}Q_0e^{-it\Delta}-i\int_0^t e^{i(t-s)\Delta}[\rho_Q, \Pi^-+Q](s)e^{-i(t-s)\Delta}ds,
\end{equation}
we get
\begin{equation}
\|Q(t)\|_{\mathfrak{H}^2}\leq \|Q_0\|_{\mathfrak{H}^2}+ c\int_0^t \Big\{1+\|\rho_Q(s)\|_{H^2}\Big\}\|Q(s)\|_{\mathfrak{H}^2}ds.
\end{equation}
Hence, by the Gronwall inequality, we obtain
\begin{equation}
\begin{aligned}\|Q(t)\|_{\mathfrak{H}^2}&\leq \|Q_0\|_{\mathfrak{H}^2}\exp\Big\{c\int_0^t (1+\|\rho_Q(s)\|_{H^2})ds\Big\}\\
&\leq\|Q_0\|_{\mathfrak{H}^2}\exp\Big\{c|I|+c|I|^{\frac{1}{2}}\|\rho_Q\|_{L_t^2(I;H^2)}\Big\}\quad\textup{(by $|I'|\leq |I|$)}\\
&<\infty\quad\textup{(by \eqref{existence time: claim})}.
\end{aligned}
\end{equation}
Thanks to this a priori bound,  
we can extend the interval $I'$ of existence to $I$, since $Q(t)$ does not blow up in $\mathfrak{H}^2$ on the interval $I$. Therefore, we conclude that $Q(t)\in C_t(I;\mathfrak{H}^2)$.

\textit{Step 2.} In order to prove \eqref{existence time: claim}, we choose intervals $I$ and $I'$ as follows. Let $I$ be the interval of existence in the proof of Theorem \ref{LWP0} with $\alpha=1$. 
Then 
\begin{align*} 
&\|Q\|_{\mathfrak{Y}^1(I)} \leq2cA, \\
&|I| =\min\{\frac{1}{(4c)^2}, \frac{1}{(8c^2A)^8}\}), 
\end{align*}
with 
$$
	A=\|Q_0\|_{\mathfrak{X}^1}+\{\textup{Tr}(-\Delta-1)Q_0\}^{\frac{1}{2}}.
$$ 
Let $I'\subset I$ be a shorter interval such that the solution $Q(t)$ exists in $\mathfrak{Y}^{\frac{3}{2}-\eta}(I')$ with small $\eta>0$. We claim that there exists $C_1>0$, depending on $I$, and not on $I'$, such that
\begin{equation}\label{rho Q 3/2-eta}
\|\rho_Q\|_{L_t^2(I'; H^{2-\eta})}\leq\|Q\|_{\mathfrak{Y}^{\frac{3}{2}-\eta}(I')}\leq C_1<\infty.
\end{equation}
Indeed, applying Strichartz estimates to the Duhamel formula which leads \eqref{LWP proof 2} with $\alpha=\frac{3}{2}-\eta$, we get
\begin{equation}
\|Q\|_{\mathfrak{Y}^{\frac{3}{2}-\eta}(I')}\leq cA'+c|I'|^{\frac{1}{2}}\|\rho_Q\|_{L_t^2(I';H^{\frac{3}{2}-\eta})}+c|I'|^{\frac{1}{8}}\|\rho_Q\|_{L_t^2(I';H^{\frac{3}{2}-\eta})} \|Q\|_{\mathcal{S}^{\frac{3}{2}-\eta}(I')},
\end{equation}
where 
$$
A'=\|Q_0\|_{\mathfrak{X}^{\frac{3}{2}-\eta}}+\{\textup{Tr}(-\Delta-1)Q_0\}^{\frac{1}{2}}.
$$ 
Hence, it follows from the choice of $I'(\subset I)$ and the definition of the $\mathfrak{Y}^\alpha$-norm (see \eqref{Y alpha norm}) that
\begin{equation}
\begin{aligned}
\|Q\|_{\mathfrak{Y}^{\frac{3}{2}-\eta}(I')}&\leq cA'+c|I|^{\frac{1}{2}}\|\rho_Q\|_{L_t^2(I;H^{\frac{3}{2}-\eta})}+c|I|^{\frac{1}{8}}\|\rho_Q\|_{L_t^2(I;H^{\frac{3}{2}-\eta})} \|Q\|_{\mathcal{S}^{\frac{3}{2}-\eta}(I')}\\
&\leq cA'+c|I|^{\frac{1}{2}}\|Q\|_{\mathfrak{Y}^1(I)}+c|I|^{\frac{1}{8}}\|Q\|_{\mathfrak{Y}^1(I)} \|Q\|_{\mathfrak{Y}^{\frac{3}{2}-\eta}(I')}\\
&\leq cA'+\frac{cA}{2}+\frac{1}{4}\|Q\|_{\mathfrak{Y}^{\frac{3}{2}-\eta}(I')}.
\end{aligned}
\end{equation}
Hence, $\|Q\|_{\mathfrak{Y}^{\frac{3}{2}-\eta}(I')}\leq2 cA'=C_1$.

Next, we  improve \eqref{rho Q 3/2-eta} to $\|Q\|_{\mathfrak{Y}^{\frac{3}{2}^+}(I')}<\infty$. Let $I''\subset I'$ be a sufficiently short interval such that the solution $Q(t)$ exists in $\mathfrak{Y}^{\frac{3}{2}^+}(I'')$. For $t\in I''$, applying trivial estimates to the Duhamel formula for $Q$,
\begin{equation}
\|\Pi_2^+Q(t)\|_{\mathcal{H}^{\frac{3}{2}^+}}\leq\|Q_0\|_{\mathcal{H}^{\frac{3}{2}^+}}+2\int_0^t \|\rho_Q(s)\Pi^-\|_{\mathcal{H}^{\frac{3}{2}^+}}+\|\rho_QQ(s)\|_{\mathcal{H}^{\frac{3}{2}^+}}ds.
\end{equation}
Using fractional Leibnitz rule (Theorem A.8 in \cite{KPV}) as in the proof of Lemma \ref{LWP lemma}, one can show that
\begin{equation}\label{persistence proof1}
\|\rho_Q(s)\Pi^-\|_{\mathcal{H}^{\frac{3}{2}^+}}=\|\la\nabla_{x'}\ra^{\frac{3}{2}^+}(\rho_Q(x)(\la\nabla\ra^{\frac{3}{2}^+}\Pi^-)(x-x'))\|_{L_{x}^2L_{x'}^2}\lesssim \|\rho_Q\|_{H^{\frac{3}{2}^+}}.
\end{equation}
Similarly, 
\begin{equation}
\begin{aligned}
\|\rho_QQ\|_{\mathcal{H}^{\frac{3}{2}^+}}&\leq \|\rho_Q\Pi_2^-Q\Pi_2^-\|_{\mathcal{H}^{\frac{3}{2}^+}}+\|\rho_Q(Q-\Pi_2^-Q\Pi_2^-)\|_{\mathcal{H}^{\frac{3}{2}^+}}\\
&=\|\la\nabla\ra^{\frac{3}{2}^+}\rho_Q\Pi_2^-Q\Pi_2^-\la\nabla\ra^{\frac{3}{2}^+}\|_{\mathfrak{S}^2}\\
&\quad+\|\la\nabla_{x}\ra^{\frac{3}{2}^+}(\rho_Q(x)\la\nabla_{x'}\ra^{\frac{3}{2}^+}(Q-\Pi_2^-Q\Pi_2^-)(x,x'))\|_{L_{x}^2L_{x'}^2}\\
&\lesssim\|\la\nabla\ra^{\frac{3}{2}^+}\rho_Q\Pi_2^-\|_{\mathfrak{S}^2}\|Q\|_{\textup{Op}}\|\Pi_2^-\la\nabla\ra^{\frac{3}{2}^+}\|_{\textup{Op}}\\
&\quad+\|\rho_Q\|_{H^{\frac{3}{2}^+}}\|\la\nabla_{x'}\ra^{\frac{3}{2}^+}(Q-\Pi_2^-Q\Pi_2^-)(x,x')\|_{L_{x}^\infty L_{x'}^2}\\
&\quad+\|\rho_Q\|_{L^\infty}\|\la\nabla_{x}\ra^{\frac{3}{2}^+}\la\nabla_{x'}\ra^{\frac{3}{2}^+}(Q-\Pi_2^-Q\Pi_2^-)(x,x')\|_{L_{x}^2L_{x'}^2}.
\end{aligned}
\end{equation}
Then, by \eqref{persistence proof1} for $\la\nabla\ra^{\frac{3}{2}^+}\rho_Q\Pi_2^-$ and the Sobolev inequality for $Q-\Pi_2^-Q\Pi_2^-=\Pi_2^+Q+\Pi_2^- Q\Pi_2^+$,
\begin{equation}\label{persistence proof2}
\begin{aligned}
\|\rho_QQ\|_{\mathcal{H}^{\frac{3}{2}^+}}&\lesssim\|\rho_Q\|_{H^{\frac{3}{2}^+}}+\|\rho_Q\|_{H^{\frac{3}{2}^+}}\|Q-\Pi_2^-Q\Pi_2^-\|_{\mathcal{H}^{\frac{3}{2}^+}}\\
&\leq\|\rho_Q\|_{H^{\frac{3}{2}^+}}\Big\{1+\|\Pi_2^+Q\|_{\mathcal{H}^{\frac{3}{2}^+}}+\|\Pi_2^- Q\Pi_2^+\|_{\mathcal{H}^{\frac{3}{2}^+}}\Big\}\\
&\leq\|\rho_Q\|_{H^{\frac{3}{2}^+}}\Big\{1+2\|\Pi_2^+Q\|_{\mathcal{H}^{\frac{3}{2}^+}}\Big\}\quad\textup{(by symmetry)}.
\end{aligned}
\end{equation}
Thus, it follows that 
\begin{equation}
\|\Pi_2^+Q(t)\|_{\mathcal{H}^{\frac{3}{2}^+}}\leq\|Q_0\|_{\mathcal{H}^{\frac{3}{2}^+}}+c\int_0^t \|\rho_Q(s)\|_{H^{\frac{3}{2}^+}}\Big\{1+2\|\Pi_2^+Q(s)\|_{\mathcal{H}^{\frac{3}{2}^+}}\Big\}ds.
\end{equation}
Therefore, by Gronwall's inequality, we conclude that
\begin{equation}\label{persistence proof3}
\begin{aligned}
\|\Pi_2^+Q(t)\|_{\mathcal{H}^{\frac{3}{2}^+}}&\leq(1+\|Q_0\|_{\mathcal{H}^{\frac{3}{2}^+}})\exp\Big\{\int_0^t c\|\rho_Q(s)\|_{H^{\frac{3}{2}^+}}ds\Big\}\\
&\leq(1+\|Q_0\|_{\mathcal{H}^{\frac{3}{2}^+}})\exp\Big\{c|I|^{1/2}\|\rho_Q\|_{L_t^2(I;H^{\frac{3}{2}^+})}\Big\}\quad\textup{(by $I''\subset I$)}\\
&\leq(1+\|Q_0\|_{\mathcal{H}^{\frac{3}{2}^+}})\exp\Big\{c|I|^{1/2}C_1\Big\}=:C_2<\infty\quad\textup{(by \eqref{rho Q 3/2-eta})}
\end{aligned}
\end{equation}
Finally, applying Strichartz estimates (Theorem \ref{Strichartz estimates for operator kernels}) to the Duhamel formula as in the proof of Theorem \ref{LWP0} and then using \eqref{persistence proof1} and \eqref{persistence proof2}, we get
\begin{equation}\label{persistence proof4}
\begin{aligned}
\|Q\|_{\mathfrak{Y}^{\frac{3}{2}^+}(I'')}&\lesssim A''+\int_0^t \|\rho_Q(s)\Pi^-\|_{\mathcal{H}^{\frac{3}{2}^+}}+\|\rho_QQ(s)\|_{\mathcal{H}^{\frac{3}{2}^+}}ds\\
&\lesssim A''+\int_0^t \|\rho_Q(s)\|_{H^{\frac{3}{2}^+}}\Big\{1+2\|\Pi_2^+Q(s)\|_{\mathcal{H}^{\frac{3}{2}^+}}\Big\}ds\\
&\leq A''+|I|^{1/2}\|\rho_Q\|_{L_t^2(I;H^{\frac{3}{2}^+})}(1+C_2)\quad\textup{(by \eqref{persistence proof3})}\\
&\leq A''+|I|^{1/2}C_1(1+C_2)=:C_3\quad\textup{(by \eqref{rho Q 3/2-eta})},
\end{aligned}
\end{equation}
where 
$$
A''=\|Q_0\|_{\mathfrak{X}^{\frac{3}{2}^+}}+\{\textup{Tr}(-\Delta-1)Q_0\}^{\frac{1}{2}}.
$$ 
Moreover, by the definition of the norm, $\|\rho_Q\|_{L_t^2(I'';H^2)}\leq C_3$. Note that this bound is independent of $I''\subset I$, and that $Q(t)\in\mathcal{K}$ for all $t\in I''$ by the conservation of relative energy (see Step 1 with \eqref{persistence proof4}). Thus, by Theorem \ref{LWP0}, we can extend the interval $I''$ to $I$ with the desired bound $\|\rho_Q\|_{L_t^2(I;H^2)}\leq C_3$.
\end{proof}

Now we are ready to prove the relative energy conservation law.

\begin{proof}[Proof of Theorem \ref{relative energy conservation0}]
Let $Q(t)\in \mathcal{Y}(I)$ be the solution to the equation \eqref{NLS} with initial data $Q_0\in\mathcal{K}$, where $I=[-T_1, T_2]$. Let $Q^{(n)}(t)$ be the solutions with initial data $Q_0^{(n)}$, where $\{Q_0^{(n)}\}_{n=1}^\infty$ is a sequence of initial data approximating $Q_0$, obtained from Lemma \ref{approximation}. For arbitrary $\epsilon>0$, we assume that $n$ is sufficiently large so that $Q^{(n)}(t)\in C(I^\epsilon; \mathfrak{H}^2)$ with $I^\epsilon=[-T_1+\epsilon, T_2-\epsilon]$. Then, it follows from continuity of the data-to-solution map from $\mathfrak{X}^1$ to $\mathfrak{Y}^1(I^\epsilon)$ that 
\begin{equation}
\lim_{n\to\infty}\|Q^{(n)}(t)-Q(t)\|_{\mathfrak{Y}^1(I^\epsilon)}=0.
\end{equation}
In particular, we have
\begin{equation}
\lim_{n\to\infty}\|\rho_{Q^{(n)}(t)}-\rho_{Q(t)}\|_{L^\infty(I^\epsilon; L^2)}=0.
\end{equation}
Moreover, for any $t\in I^\epsilon$, by weak semi-continuity of $\textup{Tr}(-\Delta-1)Q$, we have
\begin{equation}
\textup{Tr}(-\Delta-1)Q(t)\leq \liminf_{n\to\infty}\textup{Tr}(-\Delta-1)Q^{(n)}(t).
\end{equation}
Hence, it follows that 
\begin{equation}
\begin{aligned}
\mathcal{E}(Q(t))&=\textup{Tr}(-\Delta-1)Q(t)+\frac{1}{2}\int_{\mathbb{R}^d}\big(\rho_{Q(t)}\big)^2 dx\\
&\leq \liminf_{n\to\infty}\Big\{\textup{Tr}(-\Delta-1)Q^{(n)}(t)+\frac{1}{2}\int_{\mathbb{R}^d}\big(\rho_{Q^{(n)}(t)}\big)^2 dx\Big\}.
\end{aligned}
\end{equation}
Then, using the conservation law for regular solutions and Lemma \ref{approximation}, we prove that 
\begin{equation}
\begin{aligned}
\mathcal{E}(Q(t))&\leq\liminf_{n\to\infty}\Big\{\textup{Tr}(-\Delta-1)Q_0^{(n)}+\frac{1}{2}\int_{\mathbb{R}^d}\big(\rho_{Q_0^{(n)}}\big)^2 dx\Big\}\\
&=\Big\{\textup{Tr}(-\Delta-1)Q_0+\frac{1}{2}\int_{\mathbb{R}^d}\big(\rho_{Q_0}\big)^2 dx\Big\}=\mathcal{E}(Q_0).
\end{aligned}
\end{equation}
Repeating the above backward in time, we prove that $\mathcal{E}(Q_0)\leq \mathcal{E}(Q(t))$. 
Thus, we conclude that $\mathcal{E}(Q_0)= \mathcal{E}(Q(t))$.
\end{proof}

\section{Proof of the main theorem (Theorem \ref{main theorem})} \label{sec-loctoglob}

We present the proof for positive time only. By the conservation of the relative energy, $Q(t)\in\mathcal{K}$ for all $t\in [0, T^+)$, where $T^+$ is the maximal existence time (in the positive direction) for $Q$ 
given in the Theorem \ref{LWP0} . Suppose that $T^+<\infty$. Then, we evolve $Q(t)$ from the time $T^+-\epsilon$ with sufficiently small $\epsilon>0$. However, since the size of the interval of existence has a lower bound depending only on 
$$
\textup{Tr}_0(-\Delta-1)Q(T^--\epsilon)\leq\mathcal{E}(Q(T^--\epsilon))=\mathcal{E}(Q_0),
$$
we can deduce a contradiction. 

\appendix

\section{Relative energy}

\begin{lemma}\label{relative kinetic energy'}
Let $d=2,3$. Suppose that $Q\in \mathcal{K}$.\\
$(i)$ The relative kinetic energy $\textup{Tr}_0(-\Delta-1)Q$ is positive. Moreover, 
\begin{equation}\label{kinetic part}
\textup{Tr}_0(-\Delta-1)Q=\sum_{\pm}\big\||\Delta+1|^{\frac{1}{2}}Q^{\pm\pm}|\Delta+1|^{\frac{1}{2}}\big\|_{\mathfrak{S}^1}.
\end{equation}
$(ii)$ The density function $\rho_Q$ satisfies 
\begin{equation}
\|\rho_Q\|_{L^2}\lesssim \textup{Tr}_0(-\Delta-1)Q+\Big\{\textup{Tr}_0(-\Delta-1)Q\Big\}^{\frac{1}{2}}.
\end{equation}
Therefore, the relative energy $\mathcal{E}(Q)$ is finite. 
\end{lemma}

\begin{proof}
$(i)$: \eqref{kinetic part} follows from the definition of the relative kinetic energy (see \eqref{relative kinetic energy}), since $Q^{++}=\Pi^+(\gamma-\Pi^-)\Pi^+=\Pi^+\gamma\Pi^+\geq 0$ and $Q^{--}=\Pi^-(\gamma-\Pi^-)\Pi^-=\Pi^-(\gamma-1)\Pi^-\leq 0$.

$(ii)$: Applying the Lieb-Thirring inequality in Frank, Lewin, Lieb and Seiringer \eqref{LT inequality} to $\Pi_2^-Q\Pi_2^-=\Pi_2^-\gamma\Pi_2^--\Pi^-$ and using the Taylor series expansion for $(1+x)^{1+\frac{2}{d}}-1-\frac{2+d}{d}x$, we get
\begin{equation}
\textup{Tr}_0(-\Delta-1)\Pi_2^-Q\Pi_2^-\gtrsim\int_{\mathbb{R}^d} \min\Big\{\big(\rho_{\Pi_2^-Q\Pi_2^-}\big)^2,\big(\rho_{\Pi_2^-Q\Pi_2^-}\big)^{1+\frac{2}{d}}\Big\}dx.
\end{equation}
However, since $-\Pi^-\leq\Pi_2^-Q\Pi_2^-=\Pi_2^-\gamma\Pi_2^--\Pi^-\leq\Pi_2^-$, we have $|\rho_{\Pi_2^-Q\Pi_2^-}(x)|\lesssim 1$. Hence, it follows from Lemma \ref{relative kinetic energy'} $(i)$ that
\begin{equation}\label{energy space proof 1}
\|\rho_{\Pi_2^-Q\Pi_2^-}\|_{L^2}^2\lesssim\textup{Tr}_0(-\Delta-1)Q.
\end{equation}
On the other hand, by the generalized Sobolev inequality \eqref{Sobolev inequality}, we get
\begin{equation}
\|\rho_{(Q-\Pi_2^-Q\Pi_2^-)}\|_{L^2}\lesssim\|Q-\Pi_2^-Q\Pi_2^-\|_{\mathcal{H}^1}\leq\|\Pi_2^+Q\Pi_2^+\|_{\mathcal{H}^1}+2\|\Pi_2^+Q\Pi_2^-\|_{\mathcal{H}^1}.
\end{equation}
It is obvious that by Lemma \ref{relative kinetic energy'} $(i)$,
\begin{equation}
\|\Pi_2^+Q\Pi_2^+\|_{\mathcal{H}^1}\lesssim \||\Delta+1|^{\frac{1}{2}}Q^{++}|\Delta+1|^{\frac{1}{2}}\|_{\mathfrak{S}^2}\leq\textup{Tr}_0(-\Delta-1)Q.
\end{equation}
Moreover,
\begin{equation}\label{energy space proof 2}
\|\Pi_2^+Q\Pi_2^-\|_{\mathcal{H}^1}^2\lesssim\||\Delta+1|^{\frac{1}{2}}Q\|_{\mathfrak{S}^2}^2=\textup{Tr}|\Delta+1|^{\frac{1}{2}}Q^2|\Delta+1|^{\frac{1}{2}}\leq\textup{Tr}_0(-\Delta-1)Q,
\end{equation}
because
\begin{equation}
\begin{aligned}
Q^{++}-Q^{--}-Q^2&=\Pi_1^+(\gamma-\Pi_1^-)\Pi_1^+-\Pi_1^-(\gamma-\Pi_1^-)\Pi_1^--(\gamma-\Pi_1^-)^2\\
&=\cdots=\gamma(1-\gamma)\geq 0.
\end{aligned}
\end{equation}
Collecting all, we complete the proof.
\end{proof}

\subsection*{Acknowledgements}
The work of T.C. was supported by NSF CAREER grant DMS-1151414.
Y.H. thanks Seoul National University (SNU) for their hospitality during Summer 2015, and Prof. Sanghyuk Lee for helpful discussions.
The work of N.P. was supported by NSF grant DMS-1516228. 
Also N.P. gratefully acknowledges hospitality and support from 
the Mathematical Sciences Research Institute
(MSRI) in Berkeley, California via the NSF grant  DMS-1440140.

\subsection*{Conflict of Interest}
The authors declare that they have no conflict of interest

\end{document}